\documentclass[11pt,reqno,a4paper]{amsart}
\usepackage{amsmath,amssymb,amsthm,amsaddr}
\usepackage{enumitem}
\usepackage[margin=2.5cm,top=3cm,bottom=3cm,footskip=1cm,headsep=1cm]{geometry}
\usepackage[utf8]{inputenc}
\usepackage{graphicx}
\usepackage{tikz,color}
\usepackage{float}
\usetikzlibrary{arrows}

\author{\hspace*{0em}H.~Egger$^*$, T.~Kugler$^*$, B.~Liljegren-Sailer$^\dag$, \hspace*{9em} N.~Marheineke$^\dag$, and V.~Mehrmann$^+$\hspace*{9em}}
\address{$^*$Department of Mathematics, TU Darmstadt\\$^\dag$Department of Mathematics, University Erlangen\\$^+$Inst. f. Mathematik, MA 4-5, TU Berlin, D-10623 Berlin}

\title[Model reduction for wave propagation in transport networks]{On structure-preserving model reduction\\ for damped wave propagation in transport networks}

\newtheorem{lemma}{Lemma}[section]
\newtheorem{problem}[lemma]{Problem}
\newtheorem{theorem}[lemma]{Theorem}

\theoremstyle{definition}
\newtheorem{remark}[lemma]{Remark}
\newtheorem*{example*}{Example}
\newtheorem{algorithm}[lemma]{Algorithm}

\def\div{\mathrm{div}}

\def\dt{\partial_t}
\def\dx{\partial_x}

\def\RR{\mathbb{R}}

\def\E{\mathcal{E}}

\def\G{\mathcal{G}}

\def\N{\mathcal{N}}
\def\R{\mathcal{R}}
\def\V{\mathcal{V}}
\def\Vi{{\V_{0}}}
\def\Vb{{\V_{\partial}}}
\def\CC{\mathbb{C}}
\def\VV{\mathbb{V}}
\def\WW{\mathbb{W}}
\def\ZZ{\mathbb{Z}}

\numberwithin{equation}{section}
\numberwithin{table}{section}
\numberwithin{figure}{section}

\begin{document}

\begin{abstract} 
We consider the discretization and subsequent model reduction of a system of partial differential-algebraic equations describing the propagation of pressure waves in a pipeline network. 
Important properties like conservation of mass, dissipation of energy, 
passivity, existence of steady states, and exponential stability 
can be preserved by an appropriate semi-discretization in space via a 
mixed finite element method and also during the further dimension reduction by structure 
preserving Galerkin projection which is the main focus of this paper.
Krylov subspace methods are employed for the construciton of the reduced models
and we discuss modifications needed to satisfy certain algebraic 
compatibility conditions; these are required to ensure the well-posedness of the reduced models
and the preservation of the key properties.
Our analysis is based on the underlying infinite dimensional 
problem and its Galerkin approximations. The proposed algorithms 
therefore have a direct interpretation in function spaces;
in principle, they are even applicable directly to the original system of partial differential-algebraic equations while the intermediate discretization by finite elements is only required for the actual computations. 
The performance of the proposed methods is illustrated with numerical tests
and the necessity for the compatibility conditions is demonstrated by examples.
\end{abstract}

\maketitle

\begin{quote}
\noindent 
{\small {\bf Keywords:} 
partial differential-algebraic equations,
port-Hamiltonian systems,
Galerkin projection,
structure-preserving model reduction,
passivity,
exponential stability 
}
\end{quote}

\begin{quote}
\noindent
{\small {\bf AMS-classification (2000):}
35L05, 35L50, 65L20, 65L80, 65F25, 65M60
}
\end{quote}

\section{Introduction} \label{sec:intro}

We study a system of partial differential-algebraic equations 
modeling the propagation of pressure waves in a pipeline network. The basic features of 
this problem are conservation of mass and dissipation of energy by friction 
which in turn yields passivity and exponential stability of the system 
and the convergence to unique steady states. 
All these properties can be preserved for an appropriate semi-discretization in 
space by mixed finite elements resulting in a finite dimensional differential-algebraic system with a port-Hamiltonian structure \cite{EggerKugler16b}. 
In this paper, we consider a further dimension reduction of these high dimensional models by 
structure-preserving Galerkin projection with the aim to obtain reduced models of 
smaller dimension which can be used for online simulation and control. These models 
should also yield a good approximation of the overall behavior and 
preserve the port-Hamiltonian structure and further relevant properties.

The model reduction of structured linear time-invariant systems has attracted 
significant interest in the literature, see e.g.  \cite{BaiMeerbergenSu05,Freund05,MehrmannWatkins00,Polyuga10,PolyugaVanDerSchaft10,SorensenAntoulas05} and the references given there.
Related results for second order systems have been obtained in \cite{BaiSu05a,BaiSu05b,ChahlaouiEtAl05,LohmannSalimbahrami06,MeyerSrinivasan96}, 
and the reduction of differential-algebraic equations has, for instance, been addressed in \cite{BaiFreund01,MehS05}. 
Let us refer to \cite{Antoulas05,BenMS05} for a general introduction to reduced order modeling and further references. 

It is well-known that the port-Hamiltonian structure and thus passivity of the underlying system are inherited automatically by reduced models obtained via structure-preserving Galerkin projection \cite{GugercinPolyugaRostyslavBeattieVanDerSchaft12,PolyugaVanDerSchaft11,SalimbahramiLohmannBunseGerstner08}. The preservation of further properties, like conservation of mass or uniform exponential stability, however, requires the bases of the reduced models to satisfy additional compatibility conditions which have to be guaranteed explicitly.

The reduction of infinite dimensional systems described by partial differential 
or partial differential-algebraic equations has been considered, e.g., in \cite{ChapelleGariahStainteMarie12,HinzeVolkwein05,KunischVolkwein01},
and in \cite{GruJHCTB14} the reduction of models arising in gas transport networks has been discussed.
For such problems, or discretizations thereof, 
the bases for the reduced models have to be generated by some iterative process.  
Krylov subspace methods \cite{Bai02,Freund00,Grimme97,PolyugaVanDerSchaft11} and proper orthogonal decomposition \cite{ChapelleGariahStainteMarie12,HinzeVolkwein05,KunischVolkwein01} are frequently employed for this purpose, and their analysis in a function space setting allows to obtain mesh independent results.
In this paper we consider a structure-preserving model reduction
for large scale differential-algebraic systems obtained by discretization 
of a partial differential-algebraic model. 
We utilize Krylov subspace methods for the basis construction 
together with a structure-preserving space splitting and discuss appropriate modifications 
in order to satisfy some compatibility conditions required for the proof of mass conservation, uniform exponential stability, and the existence of steady states. 
While our algorithms are formulated in an algebraic setting, they also have an 
interpretation in function spaces. This is used already for the formulation of our algorithms 
and allows a complete analysis of the reduced models. Our methods
therefore turn out to be almost independent of the intermediate finite element approximation used 
in computations and they are applicable, in principle, even directly to the underlying partial differential-algebraic system.

\medskip 

The outline of the paper is as follows: 
In the following section, we introduce the model problem under consideration 
and discuss the basic steps and arguments of our approach. The remainder of the 
manuscript is then split into three major parts: Part~I is concerned with an outline and a partial analysis of the model reduction approach and Part~II provides numerical illustration of these results. 
Part~III contains the full analysis of the reduced order models obtained with our approach
which requires us to consider the infinite dimensional problem and
its approximation by mixed finite elements. The corresponding results 
mostly follow from those in \cite{EggerKugler16b} and they are therefore presented 
in the appendix for completeness and convenience of the reader.

\section{Model problem and outline of the approach}

The purpose of this section is to introduce in detail the problem under consideration 
and to give a rough idea of our approach and of the mutual relations between the 
underlying infinite dimensional system, the large scale finite dimensional systems
arising after discretization in space, and the reduced models we are looking for.

\subsection{Model problem}

We consider the propagation of pressure waves in a one-dimensional network of pipes 
whose geometry shall be given as finite directed and connected graph $\G=(\V,\E)$ with vertices $v \in \V$ and edges $e \in \E$.
On every pipe $e$, the conservation of mass and the balance of momentum are described by 
\begin{align}
a^e \dt p^e + \dx q^e &= 0             && \text{on } e \in \E, \ t>0,  \label{eq:sys1} \\
b^e \dt q^e + \dx p^e + d^e q^e &= 0   && \text{on } e \in \E, \ t>0. \label{eq:sys2}
\end{align}
Here $p^e$, $q^e$ denote the pressure and mass flux which are functions of space and time,
the coefficients $a^e$, $b^e$ encode properties of the fluid and the pipe,
and $d^e$ models the damping due to friction at the pipe walls. 
The coefficients are assumed to be positive and, for ease of presentation, constant on every pipe $e$.
At every inner vertex $v \in \Vi$ of the graph, corresponding to a junction of several pipes $e \in \E(v)$, we require that
\begin{align}
\sum\nolimits_{e \in \E(v)}  n^e(v) q^e(v) &= 0 && \text{for all } v \in \Vi, \ t>0, \label{eq:sys4}\\
p^e(v) &= p^{e'}(v) && \text{for all } e,e' \in \E(v), \ v \in \Vi, \ t>0.  \label{eq:sys3}
\end{align}
Here $n^e(v)=\mp 1$, depending on whether the pipe $e$ starts or ends at the vertex $v$; see Figure~\ref{fig:graph}. Furthermore, $m^e(v)$, $p^e(v)$ denote the respective functions evaluated at the vertex $v$ but still depending on time.
\begin{figure}[ht!]
\begin{minipage}[c]{.7\textwidth}
\begin{center}
\begin{tikzpicture}[scale=3]
\node[circle,draw,inner sep=2pt] (v1) at (-1.87,0) {$v_1$};
\node[circle,draw,inner sep=2pt] (v2) at (-0.87,0) {$v_2$};
\node[circle,draw,inner sep=2pt] (v3) at (0,0.5) {$v_3$};
\node[circle,draw,inner sep=2pt] (v4) at (0,-0.5) {$v_4$};
\draw[->,thick,line width=1.5pt] (v1) -- node[above] {$e_1$} ++(v2);
\draw[->,thick,line width=1.5pt] (v2) -- node[above,sloped] {$e_2$} ++(v3);
\draw[->,thick,line width=1.5pt] (v2) -- node[above,sloped] {$e_3$} ++(v4);
\end{tikzpicture}
\end{center}
\end{minipage}
\caption{\label{fig:graph}Graph $\G=(\V,\E)$ with vertices $\V=\{v_1,v_2,v_3,v_4\}$ and edges $\E=\{e_1,e_2,e_3\}$
defined by $e_1=(v_1,v_2)$, $e_2=(v_2,v_3)$, and $e_3=(v_2,v_4)$. 
Consequently $\Vi=\{v_2\}$, $\Vb=\{v_1,v_3,v_4\}$, $\E(v_2)=\{e_1,e_2,e_3\}$,
and moreover $n^{e_1}(v_1)=n^{e_2}(v_2)=n^{e_3}(v_2)=-1$ and 
$n^{e_1}(v_2)=n^{e_2}(v_3)=n^{e_3}(v_4)=1$.} 
\end{figure}
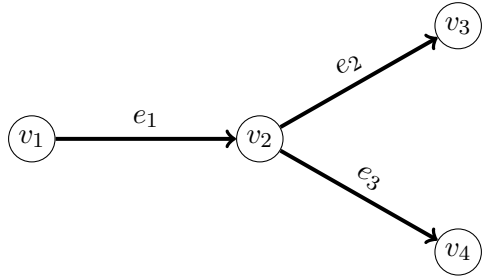
These coupling conditions model the conservation of mass and momentum at the junctions. 
At the boundary vertices $v \in \Vb = \V \setminus \Vi$, which correspond to the the ports of the network, we set
\begin{align}
p^e(v) &= u_v \qquad  \text{for } v \in \Vb, \ e \in \E(v), \ t>0 \label{eq:sys5}
\end{align}
with values $u_v$ denoting the given \emph{input} at the port $v \in \Vb$. 
As corresponding \emph{output} of the system, we consider the mass flux via the 
ports, given by
\begin{align} \label{eq:sys6}
y_v &= -n^e(v) q^e(v), \qquad v \in \Vb, \ e \in \E(v), \ t>0.
\end{align}
Other input and output configurations could be considered without difficulty as well. 
The specification of the model is completed by assuming knowledge of the initial conditions
\begin{align}
p(0)=p_0, \qquad  q(0)=q_0 \qquad \text{on } \E \label{eq:sys7}.
\end{align}

The system \eqref{eq:sys1}--\eqref{eq:sys7} models the propagation of pressure waves in a gas network on the acoustic time scale \cite{BrouwerGasserHerty11}. 
For sufficiently smooth initial data $p_0$, $q_0$, and appropriate compatible input functions $(u_v)_{v \in \Vb}$, existence of a unique classical solution can be established \cite{EggerKugler16b}.

\subsection{Basic properties}

The partial differential-algebraic system~\eqref{eq:sys1}--\eqref{eq:sys5} encodes several interesting properties 
which are directly related to the underlying physical principles:

\medskip

\medskip 
{\bf (P1) global conservation of mass}, which here can be expressed as 
 \begin{align*} 
  \frac{d}{dt} m &= \frac{d}{dt} \sum\nolimits_{e \in \E} \int_e a^e p^e dx \\
  &= -\sum\nolimits_{e \in \E} \int_e \dx q^e dx = -\sum\nolimits_{v \in \Vb} q^e(v) n^e(v) = \sum\nolimits_{v \in \Vb} y_v.  
 \end{align*}
The total mass of the gas contained in the system can thus only be altered by flow of gas into or out of the system via the ports of the network. 

\medskip 

{\bf (P2) a port-Hamiltonian structure}, leading to energy dissipation and passivity, i.e.,
 \begin{align*} 
  \frac{d}{dt} E &= \frac{d}{dt} \frac{1}{2} \sum\nolimits_{e \in \E} \int_e a^e |p^e|^2 + b^e |q^e|^2 dx  \\ & = -\sum\nolimits_{e \in \E} \int_e d^e |q^e|^2 dx + \sum\nolimits_{v \in \Vb} y_v u_v.
 \end{align*}
The total energy in the system only changes by dissipation through damping and injection or extraction via the system ports. 
Apart from these basic properties, the system further admits

\medskip 
{\bf(P3) exponential stability} and convergence to equilibrium for input $u \equiv 0$; more precisely, 
 \begin{align*} 
  E(t) \le C e^{-\gamma (t-s)} E(s), \qquad t \ge s,
 \end{align*}
with constants $C$ and $\gamma$ that are independent of the particular solution.
For constant input $u(t) \equiv const$, one can thus observe exponential convergence to 

\medskip 
{\bf (P4) unique steady states} for the corresponding stationary problem.

\medskip

\noindent
For a proof of these properties, let us refer to the appendix and to \cite{EggerKugler16b}. 
Systems of similar structure and with similar properties also model networks of electric transmission lines, vibrations of elastic multi-structures, or more general wave phenomena on multiply connected domains. Our arguments therefore may be useful in a wider context; see  \cite{BenMS05,GruJHCTB14,LagneseLeugeringSchmidt,Schilders08} for further applications.

\subsection{Full order model}

An appropriate discretization of the partial differential-algebraic system in space by mixed finite elements leads to a differential-algebraic system
\begin{alignat}{5}
M_1 \dot x_1 \ & & \            \ &+& \ G x_2 \ & & \            \ &= \ 0,    \label{eq:lti1}\\
M_2 \dot x_2 \ &-& \ G^\top x_1 \ &+& \ D x_2 \ &-& \ N^\top x_3 \ &= \ B_2 u,  \label{eq:lti2}\\
             \ & & \            \ & & \ N x_2 \ & & \            \ &= \ 0,    \label{eq:lti3}
\end{alignat}
which we will call the \emph{full order model} in the sequel.
In the context of reduced basis methods, the notion \emph{truth approximation} is sometimes used instead. 
The vectors $x_1$ and $x_2$ are the algebraic representations of the states $p$ and $q$ after discretization, and $x_3$ resembles the Lagrange multiplier for the constraint \eqref{eq:sys4}. 
The output of the system is then given by 
\begin{align}
y = B_2^\top x_2. 
\end{align}
If an appropriate discretization is used, the system matrices can be shown to have some basic structural properties. 
For ease of presentation, we formulate them here as assumptions: 
\begin{itemize}\itemsep1ex
 \item[(A0)]  $M_1$, $M_2$, $D$ are symmetric and positive definite and $[G^\top,N^\top]$
              has trivial null-space. 
\end{itemize}
The latter condition is equivalent to requiring that $G$ and the restriction of $N$ to the nullspace of $G$ define surjective linear operators. 
\begin{remark}[Notation] \label{rem:notation}
Throughout the paper, we identify matrices with corresponding linear operators. We 
call a matrix $A$ injective or surjective, if the operator has the respective property,
and we write $\R(A)$ and $\N(A)$ for the range and the kernel of the corresponding operator.
Furthermore $A \VV$ denotes the image of the space $\VV$ under the map induced by $A$.
\end{remark} 
\begin{remark} \label{rem:lti}
The differential-algebraic system \eqref{eq:lti1}--\eqref{eq:lti3} can be shown to formally have differentiation-index two  \cite{BreCP96,KunM06}.
The condition that $N^\top$ is injective, and hence that $N$ is surjective, however, allows to eliminate the Lagrange multiplier by purely algebraic manipulations and hence to reduce the system to an ordinary differential equation; see Section~\ref{sec:ode} and also refer to \cite{BenH15,EmmM13,GugSW13} 
for more general situations. Let us emphasize that the number of constraints amounts to the number of junctions in the network and thus is finite here. 
\end{remark}

\begin{remark}\label{rem:desc}
The system~\eqref{eq:lti1}--\eqref{eq:lti3} can be written 
as linear time-invariant descriptor system
\begin{align}
E \dot x + A x &= B u, \qquad y = B^\top x.
\end{align}
From the particular form of the matrices $E$ and $A$ one can directly deduce the port-Hamiltonian structure, i.e., $E$ is symmetric and positive semi-definite and $A=J+R$ can be decomposed into an skew-symmetric part $J$ and a symmetric positive semi-definite part $R$.
This immediately guarantees the passivity of the system and further useful properties \cite{SchJ14,SchM13}. 
\end{remark}

The mass and the energy of the semi-discrete system \eqref{eq:lti1}--\eqref{eq:lti3} 
can be expressed as
\begin{align}
m_h = o_1^\top M_1 x_1 
\qquad \qquad \text{and} \qquad \qquad 
E_h = \frac{1}{2} \left( x_1^\top M_1 x_1 + x_2^\top M_2 x_2 \right),
\end{align}
where $o_1$ is the vector representing the constant one function on the network. 
The basic properties (P1)--(P4) can then be shown to hold almost verbatim also for the semi-discrete problem which may therefore serve as a replacement for the infinite dimensional partial differential-algebraic problem under investigation. 

\subsection{Structure preserving model reduction} \label{sec:spmr}

The main focus of the current paper is a further dimension reduction of the differential-algebraic model \eqref{eq:lti1}--\eqref{eq:lti3} by structure-preserving Galerkin projection of the following form: 
Given projection matrices $V_1,V_2$ of appropriate size and full rank, 
we set $\widehat M_i = V_i^\top M_i V_i$, $\widehat D=V_2^\top D V_2$, 
$\widehat B_2 = V_2^\top B_2$, $\widehat G=V_2^\top G V_1$, and $\widehat N=N V_2$.
The reduced model is then defined as
\begin{alignat}{5}
\widehat M_1 \dot z_1 \ & & \            \ &+& \ \widehat G x_2 \ & & \                \ &= \ 0,    \label{eq:red1}\\
\widehat M_2 \dot z_2 \ &-& \ \widehat G^\top z_1 \ &+& \ \widehat D z_2 \ &-& \ \widehat N^\top z_3 \ &= \ \widehat B_2 u,  \label{eq:red2}\\
             \ & & \            \ & & \ \widehat N z_2       \ & & \           \ &= \ 0. \label{eq:red3}
\end{alignat}
The tuple $(V_1 z_1, V_2 z_2, z_3)$ is the approximation for the exact solution $(x_1,x_2,x_3)$ of the full order model and $\widehat y = \widehat B_2^\top z_2$ serves as approximation for the output $y = B_2^\top x_2$ of the full system.
\begin{remark} \label{lem:red}
Note that the dimension of the space for the Lagrange multiplier $x_3$ has not been reduced in the above construction; the network topology is thus completely maintained.
The reduced model can again be written in the form of a descriptor system
\begin{align}
\widehat E \dot z + \widehat A z = \widehat B u, \qquad \widehat y = \widehat B^\top \widehat z.
\end{align}
It is well-known \cite{GugercinPolyugaRostyslavBeattieVanDerSchaft12,PolyugaVanDerSchaft11,SalimbahramiLohmannBunseGerstner08} and easy to see for the system considered here that the port-Hamiltonian structure and thus passivity are inherited automatically by this kind of Galerkin projection.
Additional conditions will, however, be required to establish the well-posedness of the resulting reduced differential-algebraic system and to characterize its index; see Section~\ref{sec:basic_red}.
\end{remark}

Similar as before, we will denote by 
\begin{align}
\widehat m_h = \widehat o_1^\top \widehat M_1 z_1 
\qquad \qquad \text{and} \qquad \qquad 
\widehat E_h = \frac{1}{2} \left( z_1^\top \widehat M_1 z_1 + z_2^\top \widehat M_2 z_2 \right)
\end{align}
the mass and energy of the reduced problem \eqref{eq:red1}--\eqref{eq:red3}. 
An appropriate vector $\widehat o_1$ representing the constant one function on the network 
will be needed  and additional compatibility conditions will be required to ensure 
well-posedness of the reduced system and the validity of (P1)--(P4).

\subsection{Algebraic compatibility conditions}

As we will demonstrate by explicit examples below, the validity of some of the properties (P1)--(P4) and even the well-posedness of the reduced models can in general not be guaranteed, unless additional assumptions on the projection matrices $V_i$ are satisfied. We will therefore require that 
\begin{itemize}\itemsep1ex
 \item[(A1)] $o_1 \in \R(V_1)$;
 \item[(A2)] $\R(M_1 V_1) = \R(G V_2)$;
 \item[(A3)] $\N(G) \subset \R(V_2)$  and $N \N(G) = \R(I_3)$. 
\end{itemize}
Recall that $\R(A)$ and $\N(A)$ denote the range and the nullspace of the linear operator induced by a matrix $A$ and $A \VV$ is the the image of the space $\VV$ under mapping induced by $A$. 
Further, $I_3$ here denotes the identity matrix for the third component and $o_1$ is the vector used to describe the total mass $m_h=o_1^\top M_1 x_1$ of the full order system \eqref{eq:lti1}--\eqref{eq:lti3}.

\begin{remark}  \label{rem:A0h}
Assumption (A1) will allow us to prove the conservation of mass also for the reduced models.
The conditions (A2)--(A3), on the other hand, allow us to show that
\begin{itemize}\itemsep1ex
\item[($\widehat{\text{A0}}$)] $\widehat M_1$, $\widehat M_2$, and $\widehat D$ are symmetric and positive definite and $[\widehat G^\top,\widehat N^\top]$ has trivial nullspace;
\end{itemize}
see Lemma~\ref{lem:A0h} for details.
The reduced system thus has the same algebraic properties as the full order model \eqref{eq:lti1}--\eqref{eq:lti3}. 
The well-posedness of the reduced system \eqref{eq:red1}--\eqref{eq:red3} can therefore be obtained with similar arguments as that of the full order model. The differentiation-index of the reduced model is again two
and, by elimination of the Lagrange multiplier, we can obtain a regular system of ordinary differential equations; see Section~\ref{sec:ode} for details. 
\end{remark}

\subsection{Basis construction}

For the actual construction of the projection matrices $V_i$, we consider an extension of 
the approach proposed in \cite{Freund05} together with some modifications in order to satisfy the compatibility conditions (A1)--(A3). The main steps can be sketched as follows:
\begin{itemize}
 \item {\em Krylov iteration:} construct finite dimensional subspaces $\WW^L$ with good approximation properties by a Krylov iteration applied to the full order model \eqref{eq:lti1}--\eqref{eq:lti3}.
 \item {\em Splitting:} Decompose $\WW^L$ as $\WW^L=(\WW_1^L,0,0) + (0,\WW^L_2,0) + (0,0,\WW^L_3)$ according to the components of the state $x=(x_1,x_2,x_3)$. 
 \item {\em Modification:} choose appropriate subspaces $\ZZ_1$ and $\ZZ_2$ and define 
 \begin{align*}
  \VV_1 = \WW_1^L + \ZZ_1, \qquad \VV_2 = \WW_2^L + \ZZ_2, \qquad \text{and} \qquad \VV_3 = \R(I_3),
 \end{align*}
 such that the properties (A1)--(A3) can be verified for any choice of $V_i$, $i=1,2,3$, whose columns form bases for the corresponding subspaces.  
\end{itemize}
With similar arguments as in \cite{Freund05}, the reduced models \eqref{eq:red1}--\eqref{eq:red3} can be shown to match certain moments of the transfer function and thus to have good approximation properties. By construction, the projection matrices also satisfy the compatibility conditions (A1)--(A3).
This will allow us to show that the reduced models are well-posed and that they inherit the structural properties (P1)--(P4) from the full order model. 
\subsection{Overview}
The derivation of the properties (A0) for the system matrices of the full order model and of the algebraic compatibility conditions (A1)--(A3), as well as the complete analysis of the resulting reduced order models require us to consider in detail the connection between
\begin{itemize}\itemsep1ex
\item the underlying partial differential-algebraic equations;
\item their discretization by Galerkin approximations in a function space setting; and
\item the corresponding linear time-invariant systems in algebraic form.
\end{itemize}
A sketch of these different viewpoints is depicted in Figure~\ref{fig:sketch}.
The close relation of the differential-algebraic systems to the problem on the continuous level will allow us to establish properties of the reduced order models that are \emph{uniform} and almost \emph{independent of the intermediate full order model} which is only required for the actual computations.

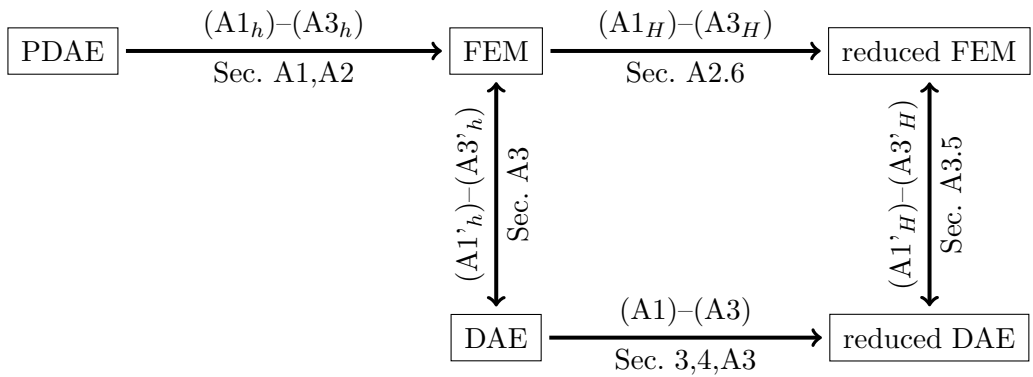
\begin{figure}[ht!]
\begin{center}
\begin{tikzpicture}[scale=1.9]
\node[draw,inner sep=5pt,outer sep=3pt] (v1) at (0,2) {PDAE};
\node[draw,inner sep=5pt,outer sep=3pt] (v2) at (3,2) {FEM};
\node[draw,inner sep=5pt,outer sep=3pt] (v3) at (6,2) {reduced FEM};
\node[draw,inner sep=5pt,outer sep=3pt] (v4) at (3,0) {DAE};
\node[draw,inner sep=5pt,outer sep=3pt] (v5) at (6,0) {reduced DAE};
\draw[->,thick,line width=1.5pt] (v1) -- node[above] {(A1$_h$)--(A3$_h$)} node[below,sloped] {Sec.~\ref{sec:properties},\ref{sec:galerkin}} ++(v2);
\draw[->,thick,line width=1.5pt] (v2) -- node[above,sloped] {(A1$_H$)--(A3$_H$)} node[below,sloped] {Sec.~\ref{sec:strucpresmodelred}} ++(v3);
\draw[<->,thick,line width=1.5pt] (v2) -- node[above,sloped] {(A1'$_h$)--(A3'$_h$)} node[below,sloped] {Sec.~\ref{sec:algebraic}} ++(v4);
\draw[<->,thick,line width=1.5pt] (v3) -- node[above,sloped] {(A1'$_H$)--(A3'$_H$)} node[below,sloped] {Sec.~\ref{sec:algebraicmodred}} ++ (v5);
\draw[->,thick,line width=1.5pt] (v4) -- node[above,sloped] {(A1)--(A3)} node[below,sloped] {Sec.~\ref{sec:prelim},\ref{sec:subspace},\ref{sec:algebraic}}  ++(v5);
\end{tikzpicture}
\end{center}
\caption{Relation between models considered in the manuscript. 
The top row represents the problems in function spaces and the bottom row the corresponding algebraic models. The properties (A1$_h$)--(A3$_h$) and so on correspond to the compatibility conditions (A1)--(A3) on different levels of our analysis.
\label{fig:sketch}} 
\end{figure}

Apart from these analytical considerations, we also investigate in detail the algorithms for the actual subspace construction on the algebraic level and we address the following issues: 
\begin{itemize}\itemsep1ex
\item The splitting step in the subspace construction turns out to be sensitive to numerical errors. 
To overcome this, we utilize a cosine-sine decomposition in the final algorithm. 
\item Round-off errors affect the validity of (A2) after the modification step outlined above. 
      We therefore take special care in the basis construction to satisfy (A2) explicitly. 
\end{itemize}
In order to reflect the functional analytic setting of the underlying infinite dimensional problem, 
we will utilize appropriate scalar products in the formulation of our algorithms on the algebraic level; see \cite{KunischVolkwein01,KunischVolkwein02} for similar approaches.
As a consequence, the vectors obtained in the basis construction process can be interpreted as functions on the continuous level which allows for a further evaluation and interpretation of the numerical results.

\section*{Part I: Model reduction}

In the following two sections, we present our model reduction approach on the algebraic level. 
We discuss in detail the construction of the reduced models and investigate their approximation properties. Futhermore, we address some algorithmic details.

\section{Structure preserving model reduction} \label{sec:prelim}

Let us first recall some basic facts about model order reduction and then informally discuss the algebraic compatibility conditions which are at the core of our model reduction approach. The basis construction algorithms on the algebraic level will then be presented in the next section.

\subsection{Model reduction basics} \label{sec:modelreduction}
Consider a general linear time-invariant descriptor system
\begin{align}\label{eq:descriptor}
 E\dot x + Ax = Bu, \qquad y = B^\top x,
\end{align}
where $E$, $A$, $B$ are given matrices, $E$ symmetric and positive semi-definite,
and $sE+A$ defines a regular matrix pencil, i.e., $sE+A$ is regular for almost every $s \in \RR$;
see \cite{BreCP96,Dai89,KunM06,Meh91} for details. 
A formal expansion of the transfer function of the system \cite{Antoulas05,BenMS05} leads to  
\begin{align} \label{eq:transfer}
 H(s) 
 &:= B^\top(s E+A)^{-1} B 
  = \sum\nolimits_{l=0}^\infty m_l(s_0-s)^l;
\end{align}
here $s_0 \in \CC$ is some given \emph{shift parameter}. 
It is not difficult to see that the generalized moments $m_l$ can be written as 
$m_l=B^\top r_l$ with vectors $r_l$ that can be computed recursively by 
\begin{align}
 (s_0 E + A) r_0      &= B,                        \label{eq:r0}\\
 (s_0 E + A) r_l &= E r_{l-1},\quad l\ge 1.   \label{eq:r1}
\end{align}
Let us denote by $\WW^L=\text{span}\{r_0,\ldots,r_{L-1}\}$ the $L$th Krylov subspace generated by this iteration. 

For any given projection matrix $V$ of appropriate dimension and maximal rank,
we define $\widehat E = V^\top E V$, $\widehat A= V^\top A V$, and $\widehat B=V^\top B$,
and consider the reduced system 
\begin{align} \label{eq:reduced}
\widehat E \dot z + \widehat A z &= \widehat B u, \qquad \widehat y = \widehat B^\top z,
\end{align}
resulting from Galerkin projection of \eqref{eq:descriptor} onto the range of the matrix $V$. 
The transfer function of this reduced model may again be expanded as 
\begin{align}
\widehat H(s)= \widehat B^\top (s \widehat E + \widehat A)^{-1} \widehat B 
= \sum\nolimits_{l=0}^\infty \widehat m_l (s_0-s)^l,
\end{align}
and the relation between the full and the reduced order model can be characterized as follows.
\begin{lemma}[Moment matching] \label{lem:matching}
Let $\WW^L \subset \R(V)$. Then $\widehat m_l = m_l$ for $l=0,\ldots,2L-1$.
\end{lemma}
\noindent 
A proof of this assertion and further results can be found in \cite{Antoulas05,BenMS05,Grimme97}.
The lemma can be interpreted as an abstract approximation result and may therefore serve as a quality indicator. 

\subsection{Basic properties of the full order model} \label{sec:basic}

The system \eqref{eq:lti1}--\eqref{eq:lti3} can be written in the
compact form \eqref{eq:descriptor} with state vector $x=(x_1,x_2,x_3)$ and with system matrices defined by
\begin{align} \label{eq:matrices}
E=\begin{bmatrix} M_1 & 0 & 0 \\ 0 & M_2 & 0 \\ 0 & 0 & 0 \end{bmatrix},
\qquad 
A=\begin{bmatrix} 0 & G & 0 \\ -G^\top & D & -N^\top \\ 0 & N & 0\end{bmatrix}, 
\qquad 
B=\begin{bmatrix} 0 \\ B_2 \\ 0 \end{bmatrix},
\end{align}
and under basic structural assumptions, the system \eqref{eq:descriptor} can be shown to be well-posed.
\begin{lemma} \label{lem:pencil}
Let (A0) hold and let $E$ and $A$ be defined as above.
Then $s E + A$ defines a regular matrix pencil.
Moreover, $sE  + A$ is regular for any $s \ge 0$ and, in particular, $A$ is regular.
\end{lemma}
\begin{proof}
First consider the case $s >0$ and set $x=(x_1,x_2,x_3)$. 
We show that $0 = y = (sE+A) x$ implies $x=0$. 
Multiplication with $x^\top$ from the left yields
\begin{align*}
0 = x^\top (sE+A) x = s (x_1^\top M_1 x_1 + x_2^\top M_2 x_2) + x_2^\top D x_2.
\end{align*}
Since $M_i$ and $D$ are positive definite, this implies that $x_1=0$ and $x_2=0$. 
But then 
\begin{align*}
0 = y_2 = s M_2 x_2 - G^\top x_1 + D x_2 - N^\top x_3 = -N^\top x_3.
\end{align*}
From the assumptions on $G$ and $N$, we can deduce the injectivity of $N^\top$, and hence $x_3=0$.\\
Now consider the case $s=0$: By simple rearrangement of the blocks of the matrix $A$, we obtain 
\begin{align*}
\widetilde A = 
\begin{bmatrix} A_{22} & A_{21} & A_{23} \\ A_{12} & A_{11} & A_{13} \\ A_{32} & A_{31} & A_{33} \end{bmatrix} 
= 
\begin{bmatrix} D & -G^\top & -N^\top \\ G & 0 & 0 \\ N & 0 & 0 \end{bmatrix}. 
\end{align*}
Such a system is regular, if, and only if, $D$ is regular and $[G^\top,N^\top]$ has trivial nullspace; see~\cite{Brezzi74} for the corresponding result in infinite dimensions. 
These properties are guaranteed by assumption~(A0), which yields the invertibility of $\tilde A$ and hence also of $A$. 
\end{proof}

\begin{remark}
Only the injectivity of $N^\top$, or equivalently, the surjectivity of $N$ is required to obtain a regular matrix pencil $sE+A$ for $s>0$ and thus to establish well-posedness for the time-dependent problem. The stronger condition that $[G^\top,N^\top]$ is injective is, however, necessary to obtain regularity of the matrix $A$ and thus to ensure existence of unique steady states. 
\end{remark}

\subsection{Properties of the reduced problem} \label{sec:basic_red}

The reduced descriptor system \eqref{eq:reduced}, representing the reduced order model \eqref{eq:red1}--\eqref{eq:red3}, can be obtained by Galerkin projection of the descriptor system \eqref{eq:descriptor}, which represents \eqref{eq:lti1}--\eqref{eq:lti3}, with a projection matrix $V$ of the form
\begin{align}  \label{eq:V}
V = \begin{bmatrix} V_1 & 0 & 0 \\ 0 & V_2 & 0 \\ 0 & 0 & V_3 \end{bmatrix} \qquad \text{and} \qquad V_3=I_3.
\end{align}
Recall that $I_3$ is the identity matrix for the space of Lagrange multipliers and note again that we did not reduce the number of constraints here.
For a projection matrix $V$ of this form, 
the particular algebraic structure of the full order model is directly passed on to the reduced order model.
The compatibility conditions (A2)--(A3) further allow us to establish the structural properties corresponding to (A0) also for the system matrices of the reduced model.
\begin{lemma} \label{lem:A0h}
Let (A0) and (A2)--(A3) hold and let $V_1$, $V_2$ be injective. 
Then ($\widehat{\text{A0}}$) holds.
\end{lemma}
\begin{proof}
The conditions on $\widehat M_i$ and $\widehat D$ are clearly satisfied, if (A0) is valid and the columns of the projection matrices $V_i$ are linearly independent.
Due to assumption (A2), we can find for any vector $z_1$ 
a corresponding vector $z_2$ such that $G V_2 z_2 = M_1 V_1 z_1$. With this choice, we obtain 
\begin{align*}
(\widehat G^\top z_1)^\top z_2 = z_1^\top \widehat G z_2 
&= z_1^\top V_1^\top G V_2 z_2 
 = z_1^\top V_1^\top M_1 V_1 z_1 = z_1^\top \widehat M_1 z_1. 
\end{align*}
Since $\widehat M_i$ is symmetric positive definite, we obtain $z_1^\top \widehat M_1 z_1>0$ whenever $z_1 \neq 0$, and consequently $\widehat G^\top$ is injective,
or equivalently, $\widehat G$ is surjective. 
Using assumptions (A0) and (A3), 
we can further find for any $x_3=z_3$ a vector $x_2 \in N(G)$ and a vector $z_2$ such that 
\begin{align*}
x_3 = N x_2 = N V_2 z_2 \qquad \text{and} \qquad 0 = G x_2 = G V_2 z_2. 
\end{align*}
This shows that the restriction of $\widehat N$ to the nullspace of $\widehat G$ is surjective.
Together with $\widehat G$ being surjective this is equivalent to $[\widehat G^\top,\widehat N^\top]$ being injective. 
\end{proof}
With the same reasoning as in Lemma~\ref{lem:pencil}, we now obtain
\begin{lemma} \label{lem:pencil_red}
Let (A0) and (A2)--(A3) hold and $V_1$, $V_2$ be injective. 
Then the matrix pencil of the reduced problem is regular.
In particular, $s \widehat E + \widehat A$ is regular for $s \ge 0$ and thus $\widehat A$ is regular.
\end{lemma}

\begin{remark} 
As a consequence of Lemma~\ref{lem:pencil_red}, we see that the system \eqref{eq:red1}--\eqref{eq:red3} representing the reduced problem is well-posed and possesses unique steady states. By construction, the system also inherits the port-Hamiltonian structure and passivity. 
Furthermore, the reduced differential-algebraic system again has differentiation-index two,
and due to injectivity of $N^\top$, the constraints can be eliminated algebraically 
and the reduced system can thus again be reduced to an ordinary differential equation. 
\end{remark}

\section{Subspace and basis construction}\label{sec:subspace}

It is well understood \cite{Antoulas05,BenMS05} that the model reduction approach outlined above amounts to a Galerkin projection of the full order system onto subspaces $\VV_i = \R(V_i)$ generated by the columns of the projection matrices $V_i$, $i=1,2,3$. 
In the following, we will change between the algebraic viewpoint and that of function spaces viewpoint as convenient.

We consider a construction of the subspaces $\VV_i$ in the form
\begin{align} \label{eq:VV}
\VV_1 =  \WW_1 + \ZZ_1,
\qquad 
\VV_2 =  \WW_2 + \ZZ_2,
\qquad \text{and} \qquad
\VV_3=\R(I_3).
\end{align}
A Krylov iteration \cite{Freund05} together with an appropriate splitting is used for generation of the spaces $\WW_1$, $\WW_2$, which by Lemma~\ref{lem:matching} automatically 
ensures good approximation properties of the resulting reduced model. 
The spaces $\ZZ_1$, $\ZZ_2$, on the other hand, will be chosen 
in order to guarantee the compatibility conditions (A1)--(A3).
After definition of the subspaces $\VV_1$ and $\VV_2$, 
we also present the algorithms for the actual computation of the projection matrices $V_1$ and $V_2$. 

\subsection{Construction of the spaces $\WW_1$ and $\WW_2$}

We start with applying the Krylov subspace iteration \eqref{eq:r0}--\eqref{eq:r1} to the particular system \eqref{eq:lti1}--\eqref{eq:lti3}. 
The first step now reads
\begin{alignat}{5}
s_0 M_1 x_1^0 \ & & \            \ &+& \ G x_2^0 \ & & \                \ &= \ 0,    \label{eq:iter0a}\\
s_0 M_2 x_2^0 \ &-& \ G^\top x_1^0 \ &+& \ D x_2^0 \ &+& \ N^\top x_3^0 \ &= \ B_2,  \label{eq:iter0b}\\
             \ & & \            \ & & \  N x_2^0     \ & & \           \ &= \ 0, \label{eq:iter0c}
\end{alignat}
and for $l \ge 1$ the further iterations are defined accordingly by 
\begin{alignat}{5}
s_0 M_1 x_1^l \ & & \            \ &+& \ G x_2^l \ & & \                \ &= \ M_1 x_1^{l-1},    \label{eq:itera}\\
s_0 M_2 x_2^l \ &-& \ G^\top x_1^l \ &+& \ D x_2^l \ &+& \ N^\top x_3^l \ &= \ M_2 x_2^{l-1},  \label{eq:iterb}\\
             \ & & \            \ & & \ N x_2^l       \ & & \          \ &= \ 0. \label{eq:iterc}
\end{alignat}
As a direct consequence of Lemma~\ref{lem:pencil}, we obtain 
\begin{lemma}[Krylov subspace iteration] \label{lem:krylov}
Let the assumption (A0) be valid. 
Then the iterative scheme \eqref{eq:iter0a}--\eqref{eq:iterc} 
is well-defined for any $s_0 \ge 0$ and all $l \ge 0$.
\end{lemma}
Let us denote the subspaces spanned by the iterates $x_i^l$ generated by \eqref{eq:iter0a}--\eqref{eq:iterc} with
\begin{align} \label{eq:VVil}
\WW_i^{L} &= \text{span}\{x_i^0,\ldots,x_i^{L-1}\}, \qquad i=1,2.
\end{align}
The special structure of the iteration allows us to show the following result.
\begin{lemma}[Properties of subspaces] \label{lem:subspace}
Let $s_0>0$, then $M_1 \WW_1^L = G \WW_2^L$,
while for $s=0$ we have $M_1 \WW_1^{L} = G \WW_2^{L+1}$.
Moreover, $N x_2 = 0$ for any $x_2 \in \WW_2^L$ and any choice of $s_0 \ge 0$.
\end{lemma}
\begin{proof}
First consider $s_0>0$:  Then from \eqref{eq:iter0a}, 
we see that $M_1 x_1^0 = - G x_2^0$, and thus the first assertion holds true for $L=0$. 
By \eqref{eq:itera} and induction, we obtain $M_1 \WW_1^L = G \WW_2^L$ for all $L \ge 1$. 
Now consider the case $s_0=0$:
Then from \eqref{eq:iter0a}, one can deduce that $x_2^0 \in \N(G)$.  
Using \eqref{eq:itera} with $L=1$, we further see that $M_1 x_1^{0} = G x_2^{1}$,
and hence $M_1 \WW_1^1 = G \WW_2^2$. By  \eqref{eq:itera} and induction, we
finally obtain the result for all $L \ge 1$ again.
\end{proof}

\subsection{Construction of the spaces $\ZZ_1$ and $\ZZ_2$}

The structural assumption (A0) particularly implies that $G$ is surjective which allows us to ensure the following property. 
\begin{lemma} \label{lem:pseudo}
Let assumption (A0) hold. Then for any choice of $g_1$ and $g_3$, the system 
\begin{align*}
G x_2 = g_1 \qquad \text{and} \qquad N x_2 = g_3
\end{align*}
has at least one solution $x_2$. 
We denote by $x_2^\dag$ the unique solution that also
minimizes $x_2^\top M_2 x_2$. 
This minimum-norm solution $x_2^\dag$ can be expressed as 
\begin{align*}
x_2^\dag = \begin{bmatrix} G \\ N \end{bmatrix}^\dag \begin{bmatrix} g_1 \\ g_3 \end{bmatrix}
\qquad \text{with} \qquad
\begin{bmatrix} G \\ N\end{bmatrix}^{\dag} 
= M_2^{-1} \begin{bmatrix}G^\top & N^\top\end{bmatrix} \left(\begin{bmatrix} G \\ N \end{bmatrix} M_2^{-1} \begin{bmatrix}G^\top & N^\top\end{bmatrix}\right)^{-1}
\end{align*}
denoting the pseudo-inverse with respect to the scalar product induced by $M_2$. 
\end{lemma}

We can now define $o_2$ as the minimum-norm solution of $G o_2 = M_1 o_1$ and $N o_2=0$ 
and set
\begin{align} \label{eq:ZZ}
\ZZ_1 = \text{span}\{o_1\}  \qquad \text{and} \qquad 
\ZZ_2 = \text{span}\{o_2\} + \N(G) \qquad \text{for } s_0>0.
\end{align}
If $s_0=0$, then we choose $\ZZ_2 = \text{span}\{o_2\} + \N(G) + \text{span}\{x_2^L\}$ instead.
As a consequence of this construction and the observations made in Lemma~\ref{lem:matching} and \ref{lem:subspace}, we obtain the following result. 
\begin{lemma}[Compatibility] \label{lem:spaces}
Define $\VV_1 = \WW_1^{L} + \ZZ_1$ and $\VV_2 = \WW_2^L + \ZZ_2$, 
and let $V_i$ denote matrices whose columns form bases for $\VV_i$, $i=1,2$, orthogonal with respect to the scalar products induced by the matrices $M_i$, respectively.
Then the assumptions (A1)--(A3) are satisfied and the reduced model \eqref{eq:red1}--\eqref{eq:red3} satisfies the moment matching conditions of Lemma~\ref{lem:matching}.
\end{lemma}

The assertions follow directly from the construction.
Let us finally mention an equivalent representation for the space $\VV_2$ 
which will be used in our basis construction algorithm in Section~\ref{sec:algo}.
\begin{lemma} \label{lem:subspace2}
Let $\VV_i$ and $V_i$, $i=1,2$ be defined as in Lemma~\ref{lem:spaces}. Then 
\begin{align*}
\VV_2 = \R(V_2) 
= \R\left(\begin{bmatrix} G \\ N \end{bmatrix}^\dag \begin{bmatrix} M_1 V_1 \\ 0 \end{bmatrix}\right) + \N(G).
\end{align*}
\end{lemma}
\begin{proof}
The assertion follows by construction of the space $\VV_2$ and Lemmas~\ref{lem:subspace} and \ref{lem:pseudo}. 
\end{proof}

\section{Algorithms for the basis construction} \label{sec:algo}

Following the above considerations, we can now formulate algorithms for the explicit construction of the projection matrices $V_i$. 
We will use {\sc Matlab} notation throughout this section.
Let us start with the construction of the Krylov subspaces, for which we use an Arnoldi method. 

\begin{algorithm}[Krylov iteration] \label{algo:basis} $ $ 
\small
\begin{verbatim}
      % function W=krylov(E,A,B,s0,L,tol)
      r = (s0*E + A)\B;  
      r = ortho(r,[],E,tol);
      W = r;
      for l=1:L-1
          r = (s0*E + A)\(E*r);
          r = ortho(r,W,E,tol);
          W = [W,r];
      end
\end{verbatim}
\end{algorithm}
\noindent
The orthogonalization can be realized efficiently via a modified Gram-Schmidt process with one re-orthogonalization step \cite{GolubVanLoan}. Note that orthogonality is understood with respect to the bilinear form induced by the matrix $E$ here, which also appears in the definition of the energy of the system. 
The corresponding algorithm reads 

\begin{algorithm}[Orthogonalization] \label{algo:ortho} $ $ 
\small
\begin{verbatim}
      % function V=ortho(V,W,E,tol)
      for k = 1:size(V,2)
          % orthogonalize to Wj 
          for r=1:2 % use reorthonormalization
              for j = 1:size(W,2)
                  hk1j = W(:, j)' * E * V(:, k);
                  V(:, k) = V(:, k) - W(:, j) * hk1j; 
              end
          end
          % orthogonalize to previous Vj 
          for r=1:2 
              for j = 1:k-1 
                  if d(j)<tol, continue; end
                  hk1j = V(:, j)' * E * V(:, k);
                  V(:, k) = V(:, k) - V(:, j) * hk1j; 
              end
          end
          % normalize
          d(k) = sqrt(V(:,k)' * E * V(:,k));
          if d(k)>=tol, 
             V(:, k) = V(:, k) / d(k);
          end
      end
      % only keep relevant vectors
      V = V(:,find(d>tol));
\end{verbatim}
\end{algorithm}

The next step consists in the splitting of the matrix $W=[W_1;W_2;W_3]$ corresponding to the solution components $x=[x_1;x_2;x_3]$. Note that even if the columns of $W$ are orthogonal, this will in general no longer be true 
for the columns of $W_i$ and thus some re-orthogonalization is required.
For reasons of numerical stability, we here employ the cosine-sine decomposition
\begin{align*}
\begin{bmatrix} W_1 \\ W_2 \end{bmatrix} = \begin{bmatrix} U_1 & 0 \\ 0 & U_2\end{bmatrix} \begin{bmatrix} C & 0 \\ 0 & S \end{bmatrix} X^\top,
\end{align*}
where $U_1$, $U_2$, and $X$ are orthogonal and $C$, $S$ are diagonal with entries $C_{ii}^2 + S_{ii}^2=1$; this explains the name of the decomposition. 
Note that the cosine-sine decomposition and the related generalized singular value decomposition can be computed efficiently and stably \cite{GolubVanLoan,VanLoan85}.
The following algorithm additionally takes into account non-standard scalar products.
\begin{algorithm}[Stable splitting via cosine-sine decomposition] \label{algo:cs} $ $ 
\small
\begin{verbatim}
      % function [W1,W2]=split(W1,W2,M1,M2,tol)
      % compute cholesky factorizations Mi=Ri*Ri'
      R1 = chol(M1); R2 = chol(M2);

      % compute generalized svd 
      [U1,U2,X,C,S] = gsvd(R1*W1,R2*W2);

      % eliminate dependent columns 
      kc = find(diag(C)>tol); ks = find(diag(S)>tol);
      W1 = R1\U1(:,kc);       W2 = R2\U2(:,ks);
\end{verbatim}
\end{algorithm}
For the problems under investigation, the splitting via the cosine-sine decomposition does not cause a substantial computational overhead but significantly improves the stability compared to the simple splitting $W=[W_1;W_2]$ with subsequent re-orthogonalization; see Section~\ref{sec:cs} for an illustration by numerical tests.

\medskip 

As a final step in the basis construction process, we now apply the modifications to 
ensure the algebraic compatibility conditions (A1)--(A3) which finally allow us to guarantee the properties (P1)--(P4) also for the reduced models. For this purpose, we use the following implementation.
\begin{algorithm}[Modifications] \label{algo:mod} $ $ 
\small
\begin{verbatim}
      % function [V1,V2]=modify(W1,W2,M1,M2,o1,nullG,tol)
      V1 = ortho([W1,o1],[],M1,tol);
      V2 = M2\([G',N']*(([G;N]*(M2\[G',N']))\[M1*W1;zeros(size(N,1),size(V1,2))])); 
      V2 = ortho([nullG,V2],[],M2,tol);
\end{verbatim}
\end{algorithm}
Note that the matrix $V_2$ was defined here following the considerations of Lemma~\ref{lem:subspace2}. This again does not significantly increase the overall complexity but substantially improves validity of the compatibility condition (A2) in the presence of round-off errors.

\subsection*{Summary}

The previous considerations allow us to draw the following important conclusions which describe 
the basic properties of our model reduction approach. 
\begin{theorem} \label{thm:main}
Let (A0)--(A3) be valid and let the projection matrices $V_i$ be defined with the Algorithms~\ref{algo:basis}--\ref{algo:mod}.
Then the reduced order model \eqref{eq:red1}--\eqref{eq:red3} is well-posed, conserves mass, dissipates energy,  
and has exponentially stable steady states, i.e., it satisfies (P1)--(P4).
\end{theorem}
\begin{proof}
Well-posedness and the existence of unique steady states and thus property (P4) follow from Lemma~\ref{lem:pencil_red}.
Validity of (P2) is a consequence of the structure-preserving Galerkin projection. The proof of  (P1) follows by construction and (P3) can be deduced from (P2) and (P4). 
The proof of uniform exponential stability (P3) and its independence of the intermediate full order model, however, requires a detailed analysis of the underlying partial differential-algebraic model and its Galerkin approximations which will be presented in the appendix.
\end{proof}

\section*{Part II: Numerical illustration}

In the following two sections, we first demonstrate the importance of the algebraic compatibility conditions and then illustrate our main results with some numerical tests.

\section{Comparison of reduced models} \label{sec:comparison}

In all our experiments, the full order model \eqref{eq:lti1}--\eqref{eq:lti3} is obtained by the mixed finite element discretization of the system \eqref{eq:sys1}--\eqref{eq:sys6}, as discussed in Section~\ref{sec:galerkin} and, therefore, the conditions (A0) are valid. 
Implicit Runge-Kutta methods are used for integration in time, and the time step $\tau$ is chosen sufficiently small in order to minimize the errors due to the time discretization. 
Quantities $x^k$ will denote the approximations for $x(t^k)$ at time $t^k = k\tau$. 

\subsection{Test problems}

We start with some considerations for the most simple networks consisting of a single pipe of unit length and the same pipe split into two parts. The model parameters are set to $a^e=b^e=d^e=1$. 
\begin{figure}[ht!]
\begin{minipage}[c]{.7\textwidth}
\begin{center}
\begin{tikzpicture}[scale=2]
\node[circle,draw,inner sep=2pt] (v1a) at (0,0) {$v_1$};
\node[circle,draw,inner sep=2pt] (v2a) at (2,0) {$v_2$};
\draw[->,thick,line width=1.5pt] (v1a) -- node[above] {$e_1$} ++(v2a);

\node[circle,draw,inner sep=2pt] (v1b) at (3,0) {$v_1$};
\node[circle,draw,inner sep=2pt] (v2b) at (5,0) {$v_2$};
\node[circle,draw,inner sep=2pt] (v3b) at (4,0) {$v_3$};
\draw[->,thick,line width=1.5pt] (v1b) -- node[above] {$e_1$} ++(v3b);
\draw[->,thick,line width=1.5pt] (v3b) -- node[above,sloped] {$e_2$} ++(v2b);
\end{tikzpicture}
\end{center}
\end{minipage}
\caption{Network topologies for the single pipe (TP1) and the double pipe (TP2).\label{fig:pipe}} 
\end{figure}
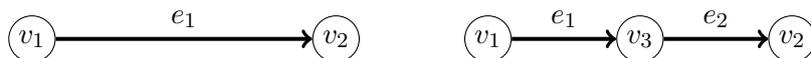
These test cases are already sufficient to illustrate the necessity of the
compatibility conditions (A1)-(A3) for the well-posedness of the reduced models and the validity of properties (P1)--(P4). 
We will compare the reduced models based on Krylov subspaces with and without modifications and refer to these as 
\begin{itemize}\itemsep1ex
 \item {the standard (reduced) model:} $\VV_1 = \WW_1^L$, $\VV_2=\WW_2^L$; and  
 \item {the improved (reduced) model:} $\VV_1 = \WW_1^L + \ZZ_1$, $\VV_2=\WW_2^L + \ZZ_2$;
\end{itemize}
respectively. 
If not stated otherwise, we always set the shift parameter to $s_0=0$. 
Moreover, we only utilize one single input for the subspace construction at the boundary vertex $v_1$ on the left side of the pipe. In this case $B_2$ consists of a single column.

\subsection{Reduction to an ordinary differential equation} \label{sec:ode}

For the test case (TP1) of a single pipe, the coupling matrix $N$ and the Lagrange parameter $x_3$ have zero dimension. Therefore, the model \eqref{eq:lti1}--\eqref{eq:lti3} is just a system of ordinary differential equations and the reduced model \eqref{eq:red1}--\eqref{eq:red2} is well-posed for any choice of projection matrices $V_1$ and $V_2$ having full rank.

This is, however, no longer true for the case (TP2) of two pipes, where $N$ and $x_3$ have one row. 
Well-posedness of the full-order system \eqref{eq:lti1}--\eqref{eq:lti3} and of the improved reduced model based on subspaces $\VV_i = \WW_i + \ZZ_i$ is still ensured by Lemmas~\ref{lem:krylov} and \ref{lem:spaces}.
For the standard reduced model based on subspaces $\VV_i = \WW_i$, on the other hand, we deduce from Lemma~\ref{lem:subspace} that 
$\widehat N = N V_2 = 0$. Hence (A3) does not hold if $\dim \VV_3 \ne 0$, and the matrix pencil $s\widehat E+\widehat A$ for the reduced model is singular. Consequently, the linear system \eqref{eq:red1}--\eqref{eq:red3} is not well-posed. 
The condition $N V_2 = 0$ can, however, be used to eliminate the Lagrange multiplier $z_3$ and to obtain the smaller system of ordinary differential equations
\begin{alignat}{5}
\widehat M_1 \dot z_1 \ & & \            \ &+& \ \widehat G x_2 \  &= \ 0, \label{eq:red1b}\\
\widehat M_2 \dot z_2 \ &-& \ \widehat G^\top z_1 \ &+& \ \widehat D z_2 \ &= \ \widehat B_2 u. \label{eq:red2b}
\end{alignat}
Whenever the standard reduced model is used in the following, we will tacitly eliminate the Lagrange multiplier and the constraints in this way.
Since the matrices $\widehat M_i$ are regular, this problem is clearly well-posed. 
Note, however, that the Lagrange multiplier $z_3$ cannot be recovered uniquely unless assumption (A3) is satisfied. 
Also, a potential input in the third equation \eqref{eq:red3}, which might occur in more general situations, cannot be handled appropriately unless this compatibility condition is valid.

\subsection{Choice of initial conditions}

Let us briefly discuss the choice of initial conditions for the reduced models. 
By projection with respect to the energy scalar products, we obtain 
\begin{align*}
z_1 = \widehat M_1^{-1} V_1^\top M_1 x_1 
\qquad \text{and} \qquad 
z_2 = \widehat M_2^{-1}V_2^\top M_2 x_2.
\end{align*}
This choice provides the best approximation of the initial conditions with respect to the energy of the problem, and the discrete energy
$E_h = \frac{1}{2} \left( x_1^\top M_1 x_1 + x_2^\top M_2 x_2\right)$ at initial time
is approximated as good as possible by the energy of the reduced model, which is obtained by replacing $x_i=V_i z_i$. Moreover, the initial energy is not increased by the projection step.
The total mass of the full order system is defined as $m_h = o_1^\top M_1 x_1$, and that of the reduced models is obtained by replacing $x_1 = V_1 z_1$ again. 

With the above choice of initial conditions, the improved reduced model will exactly reproduce the initial mass. 
For the standard reduced model, we can in general not assume that $o_1 \in \VV_1$, 
which may lead to a rather large defect in the initial mass.
As a remedy, one may enforce the correct representation of the initial mass by a constraint which, however, leads to a potential increase in the initial energy. 
In Table~\ref{tab:mass}, we display the values for the total mass obtained for initial values $p_0=1$ and $q_0=0$ with the reduced models using these two strategies.

\begin{table}[ht!]
\centering
\begin{tabular}{l|c||c||c|c|c||c|c|c}
      & & exact & \multicolumn{3}{|c||}{$\VV_i = \WW_i^L$}   & \multicolumn{3}{c}{$\VV_i = \WW_i^L + \ZZ_i$} \\
\hline
                  & $L$   &                 & $1$     & $3$     & $10$    & $1$     & $3$     & $10$    \\
\hline
projection        & $m_h$ & $1.000$ & $0.750$ & $0.902$ & $0.949$ & $1.000$ & $1.000$ & $1.000$\\ 
                  & $E_h$ & $0.500$ & $0.375$ & $0.451$ & $0.475$ & $0.500$ & $0.500$ & $0.500$\\  
\hline
mass constraint   & $m_h$ & $1.000$ & $1.000$ & $1.000$ & $1.000$ & $1.000$ & $1.000$ & $1.000$\\ 
                  & $E_h$ & $0.500$ & $0.667$ & $0.554$ & $0.527$ & $0.500$ & $0.500$ & $0.500$
\end{tabular}
\medskip
\caption{Initial values of $m_h(0)$ and $E_h(0)$ for the mass and energy for full and the reduced order models obtained by projection in the energy norm with and without additional mass constraint. Only the left input $u_1$ at vertex $v_1$ was used in the Krylov iteration and the shift parameter was set to $s_0=0$.\label{tab:mass}}
\end{table}

For the standard reduced model based on spaces $\VV_i = \WW_i^L$, we observe a substantial miss-specification of the total mass at initial time if the initial conditions are chosen by the energy projection. Exact representation of the total mass via the constraint, on the other hand, leads to an artificial increase of the initial energy. 
The size of both defects can be reduced by increasing the approximation order $L$
which allows to approximate the initial conditions better and better.
The improved reduced model, on the other hand, satisfies (A1) by construction and therefore leads to 
the exact representation of the mass and a good approximation of the energy at the same time. For the problem under investigation, the energy can even be represented exactly.

\subsection{Conservation of energy}

The port-Hamiltonian structure of the reduced order system \eqref{eq:lti1}--\eqref{eq:lti3} 
automatically leads to exact conservation of the sum of total and dissipated energy 
\begin{align*}
E_h(t) + \int_0^t D_h(s) ds = E_h(0) + \int_0^t y_h(s)^\top u(s) ds.
\end{align*}
Here $E_h=\frac{1}{2}(x_1^\top M_1 x_1 + x_2^\top M_2 x_2)$ is the total energy and $D_h = x_2^\top D x_2$ the dissipation term.
This holds true for the standard and the improved reduced model.
An additional numerical dissipation term arises if a stable implicit Runge-Kutta scheme is used for integration in time. If the time step size $\tau$ is chosen sufficiently small, the effect of this artificial numerical dissipation can however be mad arbitrarily small and thus be neglected;
more details will be given below. 

\subsection{Exponential stability}

We next consider the influence of the basis construction on the exponential stability
of the reduced models. We again set $p_0=1$ and $q_0=0$ and choose the inputs in the boundary conditions as $u_1=0$ and $u_0=1$ for $t>0$.
In Table~\ref{tab:energy}, we display the values of the energy for the full order and the reduced models for a sequence of time steps.

\begin{table}[ht!]
\centering
\begin{tabular}{l|c||c|c|c||c|c|c}
                  & exact & \multicolumn{3}{|c||}{$\VV_i = \WW_i^L$}   & \multicolumn{3}{c}{$\VV_i = \WW_i^L + \ZZ_i$} \\
\hline
 $t \setminus L$  
    &         & $1$     & $3$     & $10$    & $1$     & $3$     & $10$    \\
\hline
$0$ & $0.5000$ & $0.3750$ & $0.4512$ & $0.4745$ & $0.5000$ & $0.5000$ & $0.5000$\\ 
$1$ & $0.1528$ & $0.3750$ & $0.1665$ & $0.1514$ & $0.1876$ & $0.1552$ & $0.1527$\\ 
$2$ & $0.0512$ & $0.3750$ & $0.0708$ & $0.0509$ & $0.0690$ & $0.0511$ & $0.0511$\\ 
$3$ & $0.0174$ & $0.3750$ & $0.0384$ & $0.0173$ & $0.0245$ & $0.0173$ & $0.0174$\\ 
$4$ & $0.0059$ & $0.3750$ & $0.0273$ & $0.0059$ & $0.0084$ & $0.0059$ & $0.0059$\\ 
\end{tabular}
\medskip
\caption{Energy decay $E_h(t)$ for the full order model, the standared reduced model $\WW_i=\VV_i$, and the improved reduced model $\WW_i=\VV_i+\ZZ_i$. Only the left input $u_1$ and a shift parameter $s_0=0$ is used for the Krylov iteration.\label{tab:energy}}
\end{table}

The improved reduced model yields uniform exponential decay of the energy in all cases. Already for $L=3$, the energy is predicted accurately over the whole time interval. The standard reduced model, on the other hand, underestimates the initial energy and does not provide the correct decay rate for 
small $L$. For the smallest model with $L=1$, we do not even observe any decay in energy at all. 
As we will see next, this defect may occur for any number of moments.

\subsection{Existence of steady states}

Consider the standard reduced model: 
As shown in Section~\ref{sec:basic_red}, the unique solvability of the stationary problem requires $[\widehat G^\top,\widehat N^\top]$ to be injective. For the case (TP1) of a single pipe, the matrix $N$ has dimension zero since no inner vertex exists. Even in this case, according to Lemma~\ref{lem:subspace}, surjectivity of $\widehat G$ and thus injectivity of $\widehat G^\top$ can in general only be guaranteed for shift parameter $s_0>0$. For $s_0=0$, we expect to observe a rank deficiency in the standard reduced model and thus irregularity of the stationary problem. 

To see this, we take $s_0=0$ and let the initial values for the reduced problem be given by $z^0=(z_1^0,z_2^0,0)$ with $z_0 \in \N(\widehat A)$ and $\widehat A$ denoting the system matrix of the reduced model. Then the solution of the reduced system reads $z(t)=z^0$ for all $t \ge 0$, i.e., the damping is completely ineffective and no energy decay takes place. This can already be observed in Table~\ref{tab:energy} for $L=1$.
Note that a shift parameter $s_0=0$ is the typical choice if one is interested in the long term behavior. The standard reduced model is therefore not exponentially stable for this important case 
and existence of unique steady states cannot be guaranteed.

\section{Numerical tests for the improved reduced model} \label{sec:num}

We now illustrate in more detail the stability and approximation properties of the 
improved reduced models obtained with the algorithms proposed in the previous sections. 

\subsection{Mesh independence}

As a first example, we consider again the single pipe (TP1). 
The reduced models are generated for a single input $u_1$ at the left vertex $v_1$ 
and we set $s_0=0$.
In Figure~\ref{fig:basis}, we display the basis functions obtained for $L=4$
Krylov iterations. 

\begin{figure}[ht!]
\includegraphics[width=0.4\textwidth]{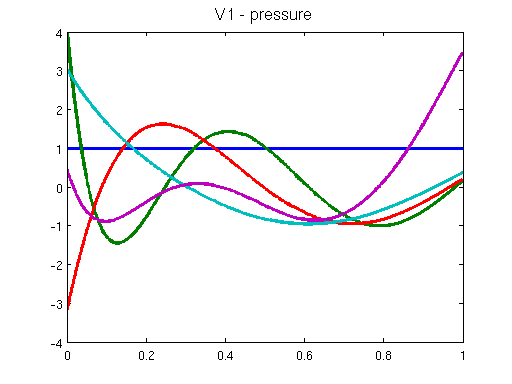}
\includegraphics[width=0.4\textwidth]{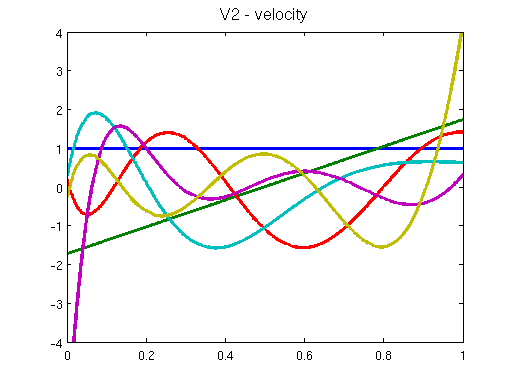}
 \caption{Bases for the subspaces $\VV_1=\WW_1+\ZZ_1$ and $\VV_2=\WW_2+\ZZ_2$ obtained by $L=4$ Krylov iterations and the modifications outlined above. 
The resulting dimensions are  $\dim(\VV_1)=5$ and $\dim(\VV_2)=6$ here.\label{fig:basis}}
\end{figure}

As explained in detail in Appendix~\ref{sec:galerkin}, the columns of the matrices $V_1$ and $V_2$ form orthogonal bases for the subspaces $\VV_1$ and $\VV_2$ which can be interpreted as functions on the interval $[0,1]$. In fact, any single Krylov iteration corresponds to the solution of an elliptic boundary value problem which explains why the basis functions are smooth. 
Also note that the functions look similar to a sequence of orthogonal polynomials of increasing degree.  
This indicates that the projection onto the Krylov subspaces leads to some sort of higher order approximation.

In the formulation of our algorithms, we payed special attention to a construction that respects the underlying function space setting. As a consequence, the subspaces $\VV_i$ and even the corresponding bases turn out to be almost independent of the underlying full order model. To illustrate this fact, we display in Figure~\ref{fig:basis-ref} one of the basis functions for velocity and pressure computed with full order models resulting from discretization on different meshes. 

\begin{figure}[ht!]
\includegraphics[width=0.4\textwidth]{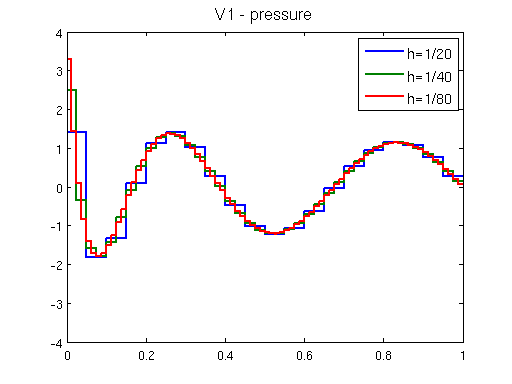}
\includegraphics[width=0.4\textwidth]{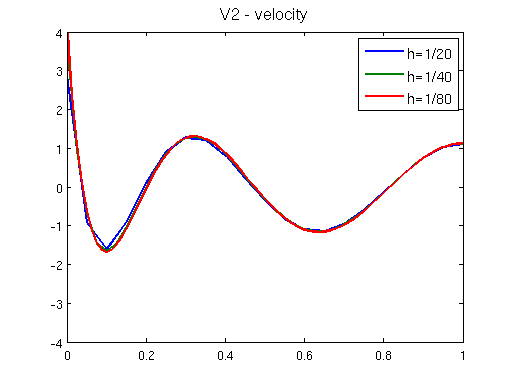}
 \caption{Basis functions for pressure and velocity
 computed with the same Krylov iteration but using different full order models obtained by finite element discretization with mesh sizes $h=\frac{1}{20}$, $\frac{1}{40}$, and $\frac{1}{80}$.\label{fig:basis-ref}}
\end{figure}

Note that the basis functions for different levels coincide almost perfectly, up to discretization errors. This clearly demonstrates the mesh independence of the proposed algorithms. 
As mentioned before, all algorithms could even be formulated directly for the infinite dimensional problem and the basis functions depicted in Figure~\ref{fig:basis-ref} thus correspond to approximations for the corresponding functions that would be obtained by the Krylov iteration in infinite dimensions.

\subsection{Stability of the splitting step} \label{sec:cs}

In our basis construction algorithms, we used the cosine-sine decomposition in order to improve the numerical stability of the splitting $W=[W_1;W_2]$. 
A simple splitting with re-orthogonalization of $W_1$ and $W_2$, on the other hand, could be realized as follows.
\begin{verbatim}
      % function [W1,W2]=simplesplit(W1,W2,M1,M2,tol)
      W1 = ortho(W1,[],M1,tol);
      W2 = ortho(W2,[],M1,tol);
\end{verbatim}
In Figure~\ref{fig:split}, we compare the results obtained by splitting the Krylov bases by this simple strategy with those obtained by means of the cosine-sine decomposition.

\begin{figure}[ht!]
\includegraphics[width=0.4\textwidth]{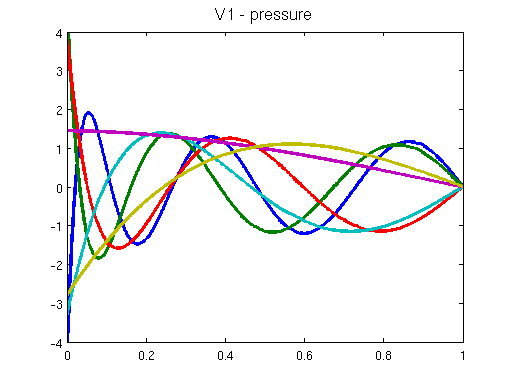}
\includegraphics[width=0.4\textwidth]{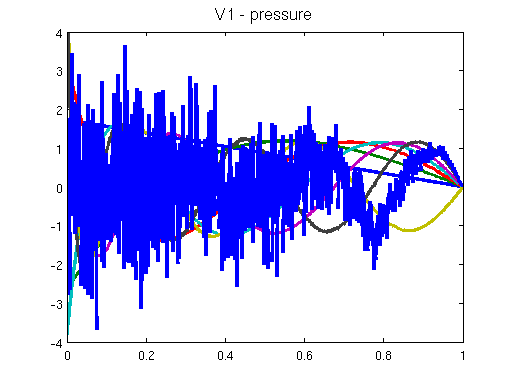}
 \caption{Basis functions for the pressure space $\WW_1^L$
 obtained after $L=10$ Krylov iterations and splitting with the cosine-sine decomposition (left) respectively the simple splitting and re-orthogonalization (right).\label{fig:split}}
\end{figure}

Due to the possibility of interpreting the basis vectors as functions on the interval $[0,1]$, one can easily conclude that,  already for relatively small dimensions, the standard splitting suffers from severe numerical instabilities. Let us emphasize that this is caused only by the instability of the splitting step and not by the Arnoldi iteration defining the Krylov spaces. The splitting via cosine-sine decomposition, on the other hand, does not suffer from these instabilities and should therefore always be preferred in practice.

\subsection{Approximation of the input-output behavior}

By Lemma~\ref{lem:matching}, we know that the reduced models obtained with our algorithms exactly match the first few moments of the transfer function. This leads to a good overall approximation of the input-output behavior in the frequency domain.
With the following tests, we would like to take a closer look also at the approximation in the time domain. 
For this purpose, we repeat the computations for a single pipe with input functions now given by
\begin{align} \label{eq:input}
u_1(t)= \begin{cases} t, & 0 \le t<1, \\ 2-t, & 1 \le t < 2, \\ 0, & t\ge 2, \end{cases}
\qquad \text{and} \qquad 
u_2(t)=0.
\end{align}
The initial values are set to $p_0=0$ and $q_0=0$ and we now use the inputs $u_1$ and $u_2$ at both pipe ends to construct the reduced models. The shift parameter is again set to $s_0=0$.
For integration of the system in time, we utilize a $\theta$-scheme with uniform time step $\tau$. For $\theta=1$ we obtain the implicit Euler method and for $\theta=\frac{1}{2}+\tau$, the scheme is second order in time. In both cases, the exponential decay is preserved \cite{EggerKugler15}. 

\begin{figure}[ht!]
\includegraphics[width=0.32\textwidth]{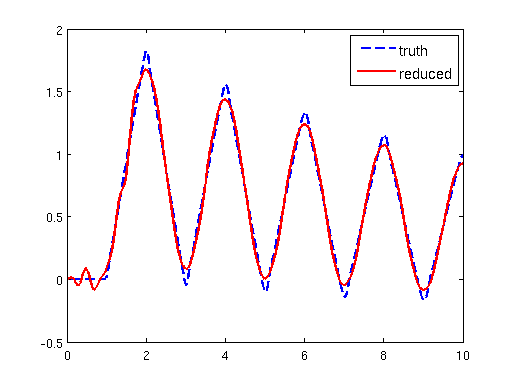} \hspace*{-1.5em}
\includegraphics[width=0.32\textwidth]{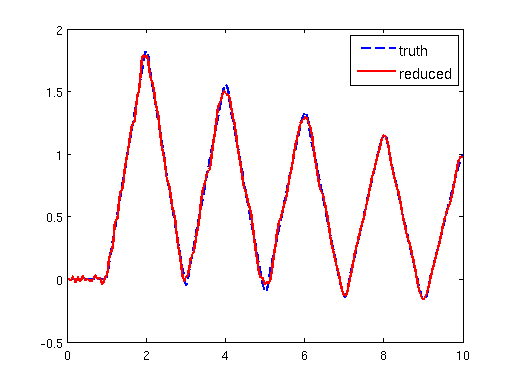} \hspace*{-1.5em}
\includegraphics[width=0.32\textwidth]{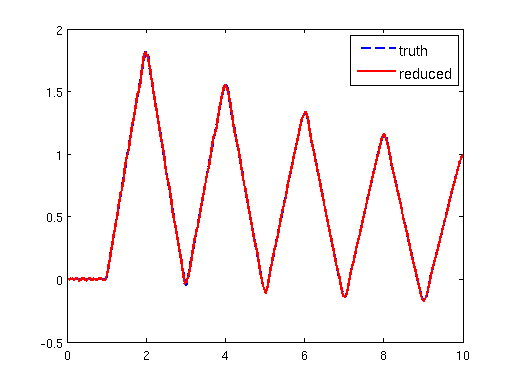}\\
\includegraphics[width=0.32\textwidth]{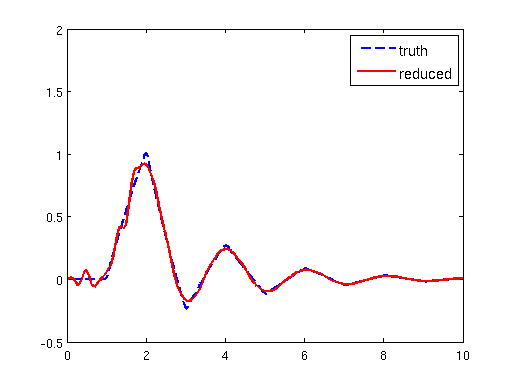} \hspace*{-1.5em}
\includegraphics[width=0.32\textwidth]{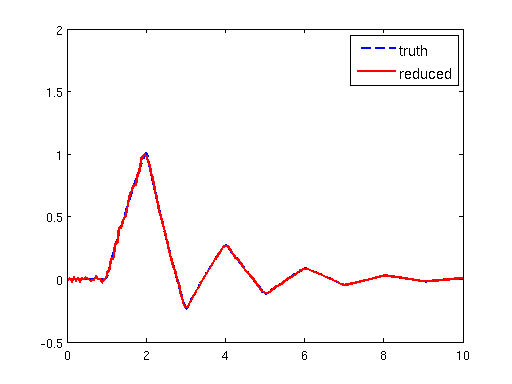} \hspace*{-1.5em}
\includegraphics[width=0.32\textwidth]{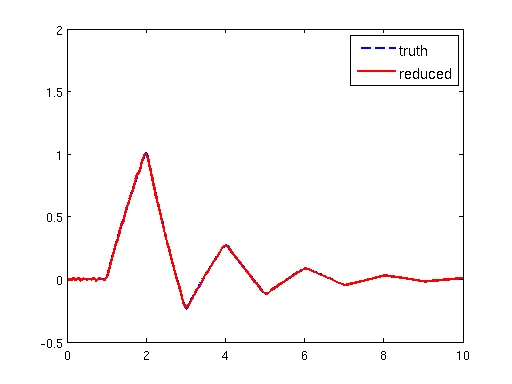}\\
\includegraphics[width=0.32\textwidth]{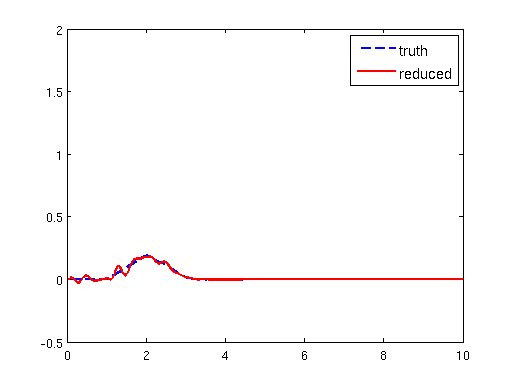} \hspace*{-1.5em}
\includegraphics[width=0.32\textwidth]{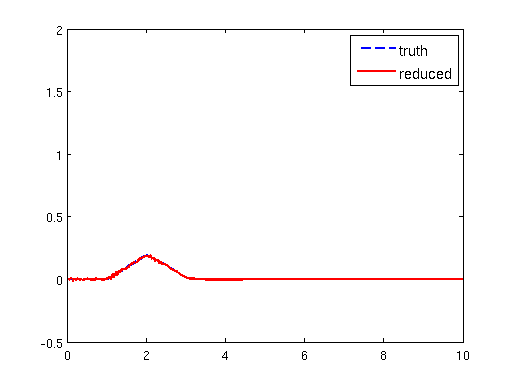} \hspace*{-1.5em}
\includegraphics[width=0.32\textwidth]{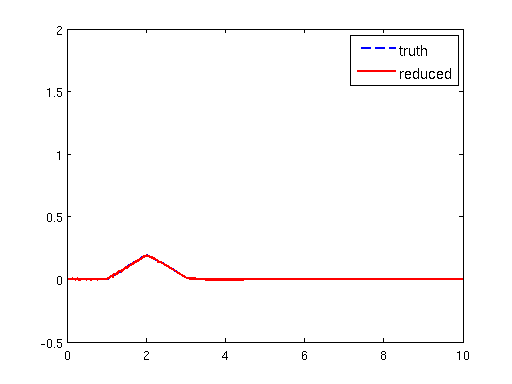}\\
\caption{Mass-flux $q(v_2)n(v_2)=-y_2$ at the right boundary for the single pipe (TP1) with input \eqref{eq:input}. Results are displayed for the full order model (blue) and reduced order models (red) with $L=2,5,10$ (left to right) and for damping parameters $d=0.1,1,5$ (top to bottom).\label{fig:impulse}}
\end{figure}

As can be seen from Figure~\ref{fig:impulse}, an approximation with only a few moments already leads to a very accurate representation of the input-output behavior in time domain. The oscillations in the initial phase of the output are due to a Gibbs phenomenon. This effect, however, becomes negligible when increasing the dimension of the reduced models. 
The correct propagation speed and damping of the signal are obtained for all models with $L \ge 5$. For the reduced model with $L=10$ Krylov iterations, which corresponds to $\dim(V_1)=14$ and $\dim(V_2)=15$ here, the prediction of the output is almost perfect. 
The plots displayed in Figure~\ref{fig:impulse} also illustrate the exponential decay of the output which becomes faster when  the damping is increased.

\subsection{Results for a small network}

As a final test case, let us now demonstrate that very similar results can also be obtained on more complicated situations. 
For this purpose we repeat the previous tests for the network depicted in Figure~\ref{fig:network}.
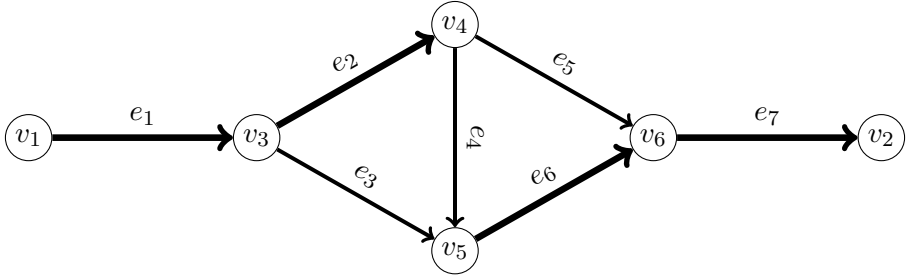
\begin{figure}[!ht]
\begin{center}
  \begin{tikzpicture}[scale=3]
  \node[circle,draw,inner sep=2pt] (v1) at (-1.87,0) {$v_1$};
  \node[circle,draw,inner sep=2pt] (v2) at (-0.87,0) {$v_3$};
  \node[circle,draw,inner sep=2pt] (v3) at (0,0.5) {$v_4$};
  \node[circle,draw,inner sep=2pt] (v4) at (0,-0.5) {$v_5$};
  \node[circle,draw,inner sep=2pt] (v5) at (0.87,0) {$v_6$};
  \node[circle,draw,inner sep=2pt] (v6) at (1.87,0) {$v_2$};
  \draw[->,thick,line width=2.5pt] (v1) -- node[above,sloped] {$e_1$} ++ (v2);
  \draw[->,thick,line width=2.5pt] (v2) -- node[above,sloped] {$e_2$} ++ (v3);
  \draw[->,thick,line width=1.5pt] (v2) -- node[above,sloped] {$e_3$} ++ (v4);
  \draw[->,thick,line width=1.5pt] (v3) -- node[above,sloped] {$e_4$} ++ (v4);
  \draw[->,thick,line width=1.5pt] (v3) -- node[above,sloped] {$e_5$} ++ (v5);
  \draw[->,thick,line width=2.5pt] (v4) -- node[above,sloped] {$e_6$} ++ (v5);
  \draw[->,thick,line width=2.5pt] (v5) -- node[above,sloped] {$e_7$} ++ (v6);
  \end{tikzpicture}
  \end{center}
\caption{Topology used for numerical tests on a network. The thickness of the edges corresponds to diameter of pipes. As before, input and output of the system occurs via the vertices $v_1,v_2$ which denote the ports of the systems.\label{fig:network}}
\end{figure}
All pipes are chosen to be of unit length $l_e=1$ and the model parameters are set constant along every pipe with values
\begin{align*}
 &a=\begin{bmatrix}4& 4& 1& 1& 1& 4& 4\end{bmatrix},\qquad b=\begin{bmatrix}1/4& 1/4& 1& 1& 1& 1/4& 1/4\end{bmatrix}\\
&\text{and} \qquad d=d_0\cdot\begin{bmatrix}1/8& 1/8& 1& 1& 1& 1/8& 1/8\end{bmatrix}.
\end{align*}
Here $d_0$ is some positive constant that allows us to vary the damping in the whole system by a single factor. 
These parameters correspond to pipes of different cross-sections; cf. Figure~\ref{fig:network}.

We now repeat the test of the previous section with input defined by \eqref{eq:input}. Since the overall system is substantially larger here, we increase the time horizon by a factor four. As before, we specify a pressure profile at the vertex $v_1$ as input and consider as output the resulting mass flux at vertex $v_2$ where the pressure is kept at zero. 
The results are depicted in Figure~\ref{fig:impulse2}.

\begin{figure}[ht!]
\includegraphics[width=0.32\textwidth]{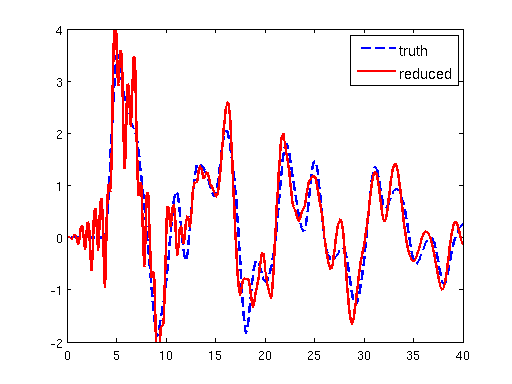} \hspace*{-1.5em}
\includegraphics[width=0.32\textwidth]{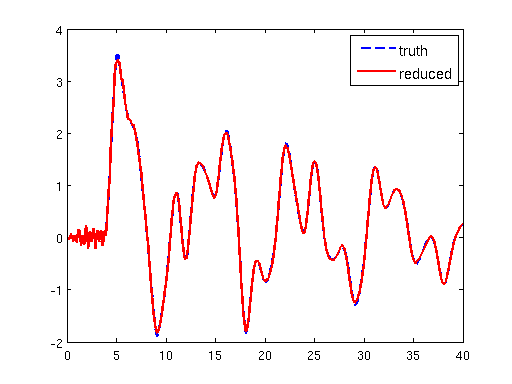} \hspace*{-1.5em}
\includegraphics[width=0.32\textwidth]{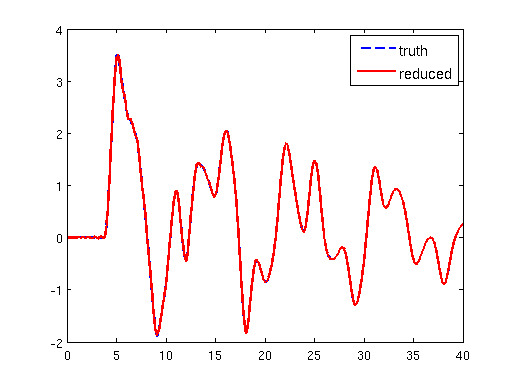}\\
\includegraphics[width=0.32\textwidth]{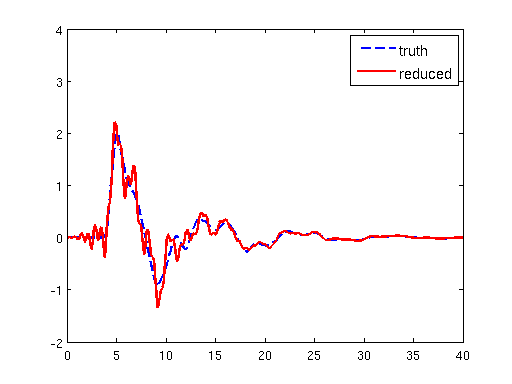} \hspace*{-1.5em}
\includegraphics[width=0.32\textwidth]{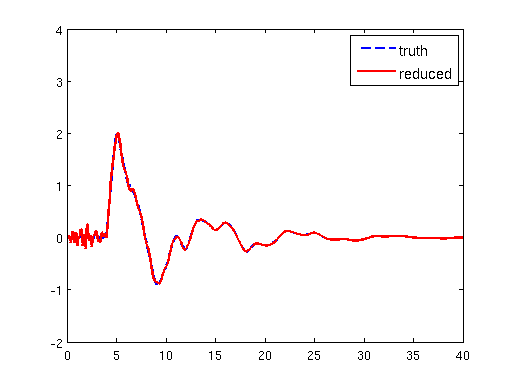} \hspace*{-1.5em}
\includegraphics[width=0.32\textwidth]{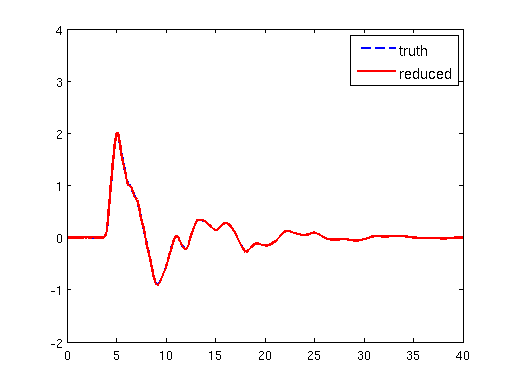}\\
\includegraphics[width=0.32\textwidth]{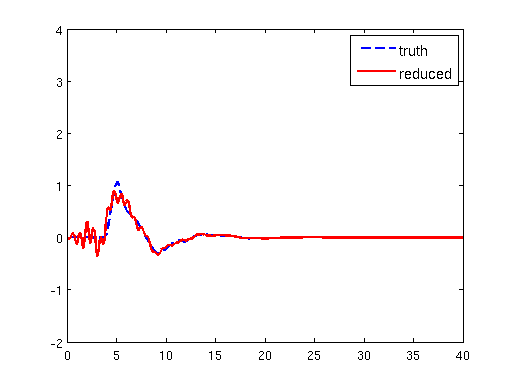} \hspace*{-1.5em}
\includegraphics[width=0.32\textwidth]{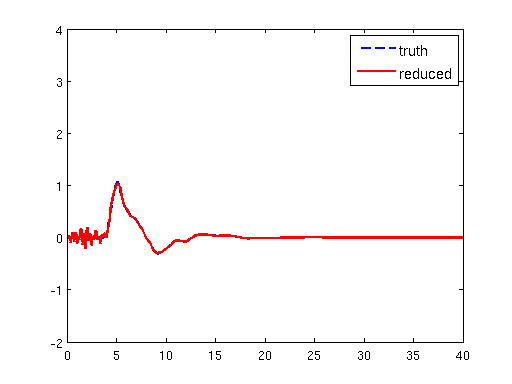} \hspace*{-1.5em}
\includegraphics[width=0.32\textwidth]{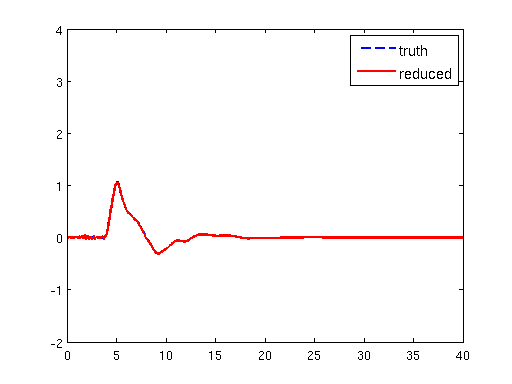}\\
\caption{Mass-flux $q(v_2)n(v_2)=-y_2$ over the right boundary for the network test problem. Results are displayed for the full order model (blue) and the reduced order models (red) with $L=5,10,20$ (left to right) and for damping parameters $d_0=0.1,0.5,1$ (top to bottom).\label{fig:impulse2}}
\end{figure}

In comparison to the example with a single pipe, the output function now has a much more complicated structure which is due to multiple pathways through the network and possible reflections at the junctions. Again we observe some Gibbs phenomena for 
approximations with only a few moments. For $L=20$, which here corresponds to $\dim(V_1)=40$ and $\dim(V_2)=47$, we already observe an almost perfect prediction of the input-output behavior. Note that this reduced model amounts to only about $6$ and $7$ degrees of freedom per pipe for pressure and velocity, respectively, which is in good agreement with the experiments for the single pipe.

\section{Discussion} \label{sec:discussion}

Let us briefly summarize the observations made in this paper.
Structure-preserving model reduction, as considered for instance in \cite{BeattieGugercin09,Freund05,GugercinPolyugaRostyslavBeattieVanDerSchaft12,MehS05,PolyugaVanDerSchaft10}, 
is in principle well suited for the systematic approximation of system of differential-algebraic equations that arise
by discretizations of partial differential-algebraic systems modeling of wave propagation phenomena on networks. 
A proper discretization and subsequent projection onto subspaces allows to preserve the underlying port-Hamiltonian structure and to guarantee passivity of the reduced models. 
Other important properties, like conservation of mass or exponential stability, are however not inherited automatically.
In order to preserve also these properties, some problem specific modifications are required in the subspace construction process. The formulation of appropriate modifications may require a detailed analysis of the underlying mathematical models in infinite dimensions and a detailed understanding of the overall discretization process
which in general problem dependent. 
Therefore, future work has to be devoted to a consideration of further applications, e.g., in elastodynamics or electromagnetics. Another aspect that has to be addressed in future research are nonlinearities in the underlying system.

\section*{Acknowledgments}
The authors are grateful for financial support by the German Research Foundation (DFG) via grants IRTG~1529 and TRR~154 project B03, C02, C04, as well as by the ``Excellence Initiative'' of the German Federal and State Governments via the Graduate School of Computational Engineering GSC~233 at Technische Universität Darmstadt.
%


\begin{thebibliography}{10}

\bibitem{Antoulas05}
A.~C. Antoulas.
\newblock {\em Approximation of large-scale dynamical systems}, volume~6 of
  {\em Advances in Design and Control}.
\newblock Society for Industrial and Applied Mathematics (SIAM), Philadelphia,
  PA, 2005.

\bibitem{Bai02}
Z.~Bai.
\newblock Krylov subspace techniques for reduced-order modeling of large-scale
  dynamical systems.
\newblock {\em Appl. Numer. Math.}, 43(1-2):9--44, 2002.

\bibitem{BaiFreund01}
Z.~Bai and R.~W. Freund.
\newblock A partial {P}ad\'e-via-{L}anczos method for reduced-order modeling.
\newblock In {\em Proceedings of the {E}ighth {C}onference of the
  {I}nternational {L}inear {A}lgebra {S}ociety ({B}arcelona, 1999)}, volume
  332/334, pages 139--164, 2001.

\bibitem{BaiMeerbergenSu05}
Z.~Bai, K.~Meerbergen, and Y.~Su.
\newblock Arnoldi methods for structure-preserving dimension reduction of
  second-order dynamical systems.
\newblock In {\em Dimension reduction of large-scale systems}, volume~45 of
  {\em Lect. Notes Comput. Sci. Eng.}, pages 173--189. Springer, Berlin, 2005.

\bibitem{BaiSu05a}
Z.~Bai and Y.~Su.
\newblock Dimension reduction of large-scale second-order dynamical systems via
  a second-order {A}rnoldi method.
\newblock {\em SIAM J. Sci. Comput.}, 26(5):1692--1709, 2005.

\bibitem{BaiSu05b}
Z.~Bai and Y.~Su.
\newblock S{OAR}: a second-order {A}rnoldi method for the solution of the
  quadratic eigenvalue problem.
\newblock {\em SIAM J. Matrix Anal. Appl.}, 26(3):640--659, 2005.

\bibitem{BeattieGugercin09}
C.~Beattie and S.~Gugercin.
\newblock Interpolatory projection methods for structure-preserving model
  reduction.
\newblock {\em Systems Control Lett.}, 58(3):225--232, 2009.

\bibitem{BenH15}
P.~Benner and J.~Heiland.
\newblock Time-dependent {D}irichlet conditions in finite element
  discretizations.
\newblock {\em ScienceOpen Research}, 2015.
\newblock in press.

\bibitem{BenMS05}
P.~Benner, V.~Mehrmann, and D.~C. Sorensen, editors.
\newblock {\em Dimension Reduction of Large-Scale Systems}, volume~45 of {\em
  Lect. Notes Comput. Sci. Eng.} Springer, 2005.

\bibitem{BreCP96}
K.~E. Brenan, S.~L. Campbell, and L.~R. Petzold.
\newblock {\em Numerical Solution of Initial-Value Problems in Differential
  Algebraic Equations}.
\newblock {SIAM} Publications, Philadelphia, PA, 2nd edition, 1996.

\bibitem{Brezzi74}
F.~Brezzi.
\newblock On the existence, uniqueness and approximation of saddle-point
  problems arising from lagrangian multipliers.
\newblock {\em RAIRO Anal. Numer.}, 2:129--151, 1974.

\bibitem{BrouwerGasserHerty11}
J.~Brouwer, I.~Gasser, and M.~Herty.
\newblock Gas pipeline models revisited: Model hierarchies, non-isothermal
  models and simulations of networks.
\newblock {\em Multiscale Model. Simul.}, 9:601--623, 2011.

\bibitem{ChahlaouiEtAl05}
Y.~Chahlaoui, K.~A. Gallivan, A.~Vandendorpe, and P.~Van~Dooren.
\newblock Model reduction of second-order systems.
\newblock In {\em Dimension reduction of large-scale systems}, volume~45 of
  {\em Lect. Notes Comput. Sci. Eng.}, pages 149--172. Springer, Berlin, 2005.

\bibitem{ChapelleGariahStainteMarie12}
D.~Chapelle, A.~Gariah, and J.~Sainte-Marie.
\newblock Galerkin approximation with proper orthogonal decomposition: new
  error estimates and illustrative examples.
\newblock {\em ESAIM Math. Model. Numer. Anal.}, 46(4):731--757, 2012.

\bibitem{Dai89}
L.~Dai.
\newblock {\em Singular Control Systems}.
\newblock Springer-Verlag, Berlin, Germany, 1989.

\bibitem{EggerKugler15}
H.~Egger and T.~Kugler.
\newblock Uniform exponential stability of {G}alerkin approximations for damped
  wave systems.
\newblock {\em arXive:1511.08341}, 2015.

\bibitem{EggerKugler16b}
H.~Egger and T.~Kugler.
\newblock Damped wave systems on networks: Exponential stability and uniform
  approximations.
\newblock Technical report, 2016.
\newblock arXive:1605.03066.

\bibitem{EmmM13}
E.~{Emmrich} and V.~{Mehrmann}.
\newblock Operator differential-algebraic equations arising in fluid dynamics.
\newblock {\em Comput. Methods Appl. Math.}, 13(4):443--470, 2013.

\bibitem{Freund00}
R.~W. Freund.
\newblock Krylov-subspace methods for reduced-order modeling in circuit
  simulation.
\newblock {\em J. Comput. Appl. Math.}, 123:395--421, 2000.
\newblock Numerical analysis 2000, Vol. III. Linear algebra.

\bibitem{Freund05}
R.~W. Freund.
\newblock Pad\'e-type model reduction of second-order and higher-order linear
  dynamical systems.
\newblock In {\em Dimension reduction of large-scale systems}, volume~45 of
  {\em Lect. Notes Comput. Sci. Eng.}, pages 191--223. Springer, Berlin, 2005.

\bibitem{GolubVanLoan}
G.~H. Golub and C.~F. Van~Loan.
\newblock {\em Matrix Computations}.
\newblock The John Hopkins University Press, Baltimore and London, 3 edition,
  1996.

\bibitem{Grimme97}
E.~J. Grimme.
\newblock {\em Krylov projection methods for model reduction}.
\newblock Dissertation, University of Illinois, Urbana-Champaign, 1997.

\bibitem{GruJHCTB14}
S.~Grundel, L.~Jansen, N.~Hornung, T.~Clees, C.~Tischendorf, and P.~Benner.
\newblock Model order reduction of differential algebraic equations arising
  from the simulation of gas transport networks.
\newblock In S.~Schöps, A.~Bartel, M.~Günther, E.~J.~W. ter Maten, and P.~C.
  Müller, editors, {\em Progress in Differential-Algebraic Equations},
  Differential-Algebraic Equations Forum, pages 183--205. Springer Berlin
  Heidelberg, 2014.

\bibitem{GugercinPolyugaRostyslavBeattieVanDerSchaft12}
S.~Gugercin, R.~V. Polyuga, C.~Beattie, and A.~van~der Schaft.
\newblock Structure-preserving tangential interpolation for model reduction of
  port-{H}amiltonian systems.
\newblock {\em Automatica J. IFAC}, 48(9):1963--1974, 2012.

\bibitem{GugSW13}
S.~Gugercin, T.~Stykel, and S.~Wyatt.
\newblock Model reduction of descriptor systems by interpolatory projection
  methods.
\newblock {\em SIAM Journal on Scientific Computing}, 35:B1010--B1033, 2013.

\bibitem{HinzeVolkwein05}
M.~Hinze and S.~Volkwein.
\newblock Proper orthogonal decomposition surrogate models for nonlinear
  dynamical systems: error estimates and suboptimal control.
\newblock In {\em Dimension reduction of large-scale systems}, volume~45 of
  {\em Lect. Notes Comput. Sci. Eng.}, pages 261--306. Springer, Berlin, 2005.

\bibitem{KunischVolkwein01}
K.~Kunisch and S.~Volkwein.
\newblock Galerkin proper orthogonal decomposition methods for parabolic
  problems.
\newblock {\em Numer. Math.}, 90:117--148, 2001.

\bibitem{KunischVolkwein02}
K.~Kunisch and S.~Volkwein.
\newblock Galerkin proper orthogonal decomposition methods for a general
  equation in fluid dynamics.
\newblock {\em SIAM J. Numer. Anal.}, 40:492--515, 2002.

\bibitem{KunM06}
P.~Kunkel and V.~Mehrmann.
\newblock {\em Differential-algebraic equations. Analysis and numerical
  solution}.
\newblock EMS Textbooks in Mathematics. European Mathematical Society (EMS),
  Z\"urich, 2006.

\bibitem{LagneseLeugeringSchmidt}
L.~E. Lagnese, G.~Leugering, and E.~J. P.~G. Schmidt.
\newblock {\em Modeling, Analysis and Control of Dynamic Elastic Multi-Link
  Structures}.
\newblock Systems \& Control: Foundations \& Applications. Springer
  Science+Business Media, New~York, 1994.

\bibitem{Meh91}
V.~{Mehrmann}.
\newblock {\em The Autonomous Linear Quadratic Control Problem. {T}heory and
  Numerical Solution}, volume 163 of {\em Lecture Notes in Control and
  Information Sciences}.
\newblock Berlin etc.: Springer-Verlag, 1991.

\bibitem{MehS05}
V.~{Mehrmann} and T.~{Stykel}.
\newblock Balanced truncation model reduction for large-scale system in
  descriptor form.
\newblock In P.~{Benner}, V.~{Mehrmann}, and D.~C. {Sorensen}, editors, {\em
  Dimension Reduction of Large-Scale Systems}, pages 83--115. Berlin: Springer,
  2005.

\bibitem{MehrmannWatkins00}
V.~Mehrmann and D.~Watkins.
\newblock Structure-preserving methods for computing eigenpairs of large sparse
  skew-{H}amiltonian/{H}amiltonian pencils.
\newblock {\em SIAM J. Sci. Comput.}, 22:1905--1925, 2001.

\bibitem{MeyerSrinivasan96}
D.~G. Meyer and S.~Srinivasan.
\newblock Balancing and model reduction for second-order form linear systems.
\newblock {\em IEEE Trans. Automat. Control}, 41(11):1632--1644, 1996.

\bibitem{Polyuga10}
R.~V. Polyuga.
\newblock Discussion on: ``{P}assivity and structure preserving order reduction
  of linear port-{H}amiltonian systems using {K}rylov subspaces''.
\newblock {\em Eur. J. Control}, 16(4):407--409, 2010.

\bibitem{PolyugaVanDerSchaft10}
R.~V. Polyuga and A.~van~der Schaft.
\newblock Structure preserving model reduction of port-{H}amiltonian systems by
  moment matching at infinity.
\newblock {\em Automatica J. IFAC}, 46(4):665--672, 2010.

\bibitem{PolyugaVanDerSchaft11}
R.~V. Polyuga and A.~van~der Schaft.
\newblock Structure preserving moment matching for port-{H}amiltonian systems:
  {A}rnoldi and {L}anczos.
\newblock {\em IEEE Trans. Automat. Control}, 56(6):1458--1462, 2011.

\bibitem{LohmannSalimbahrami06}
B.~Salimbahrami and B.~Lohmann.
\newblock Order reduction of large scale second-order systems using {K}rylov
  subspace methods.
\newblock {\em Linear Algebra Appl.}, 415(2-3):385--405, 2006.

\bibitem{SalimbahramiLohmannBunseGerstner08}
B.~Salimbahrami, B.~Lohmann, and A.~Bunse-Gerstner.
\newblock Passive reduced order modelling of second-order systems.
\newblock {\em Math. Comput. Model. Dyn. Syst.}, 14(5):407--420, 2008.

\bibitem{Schilders08}
W.~H.~A. Schilders, H.~A. van\;der\;Vorst, and J.~Rommes, editors.
\newblock {\em Model Order Reduction: Theory, Research Aspects and
  Applications}, volume~13 of {\em Mathematics in Industry}. Springer, 2008.

\bibitem{SorensenAntoulas05}
D.~C. Sorensen and A.~C. Antoulas.
\newblock On model reduction of structured systems.
\newblock In {\em Dimension reduction of large-scale systems}, volume~45 of
  {\em Lect. Notes Comput. Sci. Eng.}, pages 117--130. Springer, Berlin, 2005.

\bibitem{SchJ14}
A.~van~der Schaft and D.~Jeltsema.
\newblock Port-{H}amiltonian systems theory: An introductory overview.
\newblock {\em Foundations and Trends in Systems and Control}, 1(2-3):173--378,
  2014.

\bibitem{SchM13}
A.~van~der Schaft and B.~M. Maschke.
\newblock Port-{H}amiltonian systems on graphs.
\newblock {\em SIAM J. Control Optim.}, 51:906--937, 2013.

\bibitem{VanLoan85}
C.~Van~Loan.
\newblock Computing the {CS} and the generalized singular value decompositions.
\newblock {\em Numer. Math.}, 46:479--491, 1985.

\end{thebibliography}

\appendix

\section*{Appendix}

\renewcommand{\thesection}{A\arabic{section}}

\section*{Part III. Functional analytic background}

The purpose of this appendix is to show rigorously that the algorithms presented in the previous sections lead to reduced models that satisfy properties (P1)--(P4) uniformly. 
Most of the following results can in principle be obtained by generalization of those in \cite{EggerKugler16b}. To give a complete presentation, we repeat the most important results required for our analysis and provide short proofs where they yield further insight.

\section{The infinite dimensional problem} \label{sec:properties}

Let us start with introducing the relevant notation. We denote by
\begin{align*}
L^2(\E) = \{ p : p|_e=p^e \in L^2(e) \quad \forall e \in \E\}
\end{align*}
the space of square integrable functions over the network with norm
\begin{align*}
\|p\|_{L^2(\E)} = (p,p)_\E^{1/2} 
\quad \text{and} \quad 
(p,\tilde p)_\E = \sum\nolimits_e (p^e,\tilde p^e)_{L^2(e)}.
\end{align*}
For convenience of notation, we will sometimes use the symbols $\|\cdot\|_{L^2}$ and $\|\cdot\|$ instead.
In addition to this basic function space, we will make use of the broken Sobolev space
\begin{align*}
H^1(\E) = \{ q : q^e \in H^1(e) \quad \forall e \in \E\}
\end{align*}
consisting of functions that are continuous along edges but may be discontinuous at interior vertices $v \in \Vi$. 
The broken derivative of a function $q \in H^1(\E)$ is denoted by $\dx' q$ defined by 
\begin{align*}
(\dx' q)|_e = \dx (q|_e) \qquad \text{for all } e \in \E. 
\end{align*}
This allows us to write $H^1(\E) = \{q \in L^2(\E) : \dx' q \in L^2(\E)\}$ with natural norm defined by
\begin{align*}
\|q\|_{H^1(\E)}^2 = \|q\|_{L^2(\E)}^2 + \|\dx' q\|^2_{L^2(\E)}.
\end{align*}
For a piecewise smooth function $q \in H^1(\E)$, we define at every interior vertex $v \in \Vi$
the value
\begin{align*}
[n q](v) = \sum\nolimits_{e \in \E(v)} n^e(v) q^e(v),
\end{align*}
which amounts to the imbalance in the coupling condition \eqref{eq:sys4}.
For a junction of only two pipes, the value $[nq]$ amounts to the \emph{jump} of $q$ across the junction and the symbol $[nq]$ is the one usually employ in the analysis of discontinuous Galerkin methods. 
These values of the jumps can be understood as a vector in $\RR^{\Vi}$ and
as the scalar product on $\RR^{\Vi}$, we use 
\begin{align*}
(\lambda,\mu)_{\Vi} = \sum\nolimits_{v \in \Vi} \lambda_v \mu_v. 
\end{align*}
In a similar way, we will denote by $(\lambda,\mu)_{\Vb}$ the corresponding scalar product for $\RR^{\Vb}$.
We now have the following variational characterization of solutions to our model problem.

\begin{lemma}
Let $(p,q)$ denote a smooth solution of \eqref{eq:sys1}--\eqref{eq:sys6}
and set $\lambda_v(t)=p_v(t)$. Then
\begin{align}
(a \dt p(t), \tilde p)_\E + (\dx' q(t), \tilde p)_\E &= 0, \label{eq:var1}\\
(b \dt q(t), \tilde q)_\E - (p(t), \dx' \tilde q)_\E + (d q(t), \tilde q)_\E + (\lambda(t), [n \tilde q])_{\Vi} &= -(u(t), n \tilde q)_{\Vb}, \label{eq:var2}\\
([n q(t)], \tilde \lambda)_\Vi &= 0, \label{eq:var3}
\end{align}
for all test functions $\tilde p \in L^2(\E)$, $\tilde q \in H^1(\E)$, $\tilde \lambda \in \RR^{\Vi}$, and all $t \ge 0$.
\end{lemma}
\begin{proof}
The proof follows with similar arguments as in \cite{EggerKugler16b}. 
We therefore only sketch the required modifications:
The validity of \eqref{eq:var1} and \eqref{eq:var3} follows directly from \eqref{eq:sys1} and \eqref{eq:sys4}. 
Using integration-by-parts on one single edge $e=(v_1^e,v_2^e)$, 
we get
\begin{align*}
(\dx p^e,\tilde q^e)_e 
&= -(p^e,\dx \tilde q^e)_e + n^e(v_1^e) \tilde q^e(v_1^e) p^e(v_1^e) + n^e(v_2^e) \tilde q^e(v_2^e) p^e(v_2^e).
\end{align*}
Summing over all edges, this yields terms at the inner vertices that can be reordered as
\begin{align*}
&\sum\nolimits_{e \in \E} n^e(v_1^e) \tilde q^e(v_1^e) p^e(v_1^e) + n^e(v_2^e) \tilde q^e(v_2^e) p^e(v_2^e) \\
&\qquad \qquad \qquad 
= \sum\nolimits_{v \in \Vi} \sum\nolimits_{e \in \E(v)} n^e(v) \tilde q^e(v) p^e(v) 
+ \sum\nolimits_{v \in \Vb} n^e(v) \tilde q^e(v) p^e(v).
\end{align*}
Using that $p^e(v)=\lambda_v$ for $v \in \Vi$ and $p^e(v)=u_v$ for $v \in \Vb$ and the definition of $[n \tilde q]$, this shows the validity of \eqref{eq:var2} and completes the proof of the lemma.
\end{proof}

\begin{remark}[Well-posedness]
Together with the initial conditions \eqref{eq:sys5}, the variational problem \eqref{eq:var1}--\eqref{eq:var3} 
can be shown to admit a unique solution which, for sufficiently regular inputs and initial data,
corresponds to the classical solution of \eqref{eq:sys1}--\eqref{eq:sys6}. 
The variational formulation is thus equivalent to the initial boundary value problem;
see \cite{EggerKugler16b} for details.
\end{remark}

We now proceed by establishing the properties (P1)--(P4) on the continuous level.
The results again follow with a slight modification of the arguments given in \cite{EggerKugler16b}. 

\subsection{Conservation of mass}
The mass of the fluid contained in a single pipe is given by 
\begin{align*}
m^e(t) = \int_e a^e p^e(t) dx.
\end{align*}
Using the balance equation \eqref{eq:sys1} and the conservation condition \eqref{eq:sys4}, 
we obtain 

\begin{lemma}[Conservation of mass] \label{lem:mass} $ $ \\
Let $m(t) = \sum\nolimits_{e \in \E} m^e(t)$ denote the total mass contained in the network.
Then 
\begin{align*}
\frac{d}{dt} m(t) = \sum\nolimits_{v \in \Vb} y_v(t), 
\end{align*}
i.e., the change of mass is caused only by flux across the boundary of the network.  
\end{lemma}
\begin{proof}
Using \eqref{eq:sys1}, the fundamental theorem of calculus, and \eqref{eq:sys3}, we get 
\begin{align*}
\frac{d}{dt} \sum\nolimits_{e \in \E} m^e 
&= \sum\nolimits_{e \in \E} \int_e a^e \dt p^e dx  
 = \sum\nolimits_{e \in \E} \int_e -\dx q^e dx 
 = \sum\nolimits_{e \in \E} -n^e(v_1^e) q^e(v_1^e) - n^e(v_2^e) q^e(v_2^e) \\
&= \sum\nolimits_{v \in \V} \sum\nolimits_{e \in \E(v)} -n^e(v) q^e(v) 
 = \sum\nolimits_{v \in \Vb} -n^e(v) q^e(v).
\end{align*}
The result then follows by using the special form of the output $y_v$ given in \eqref{eq:sys7}.
\end{proof}

\subsection{Energy dissipation}
Let us first prove property (P1) by deriving an explicit energy dissipation relation.
The total acoustic energy contained in a single pipe is given by 
\begin{align*}
E^e(t) = \frac{1}{2} \int_e a^e |p^e(t)|^2 + b^e |q^e(t)|^2 dx. 
\end{align*}
From the differential equations \eqref{eq:sys1}--\eqref{eq:sys2} 
and the algebraic continuity conditions \eqref{eq:sys4}--\eqref{eq:sys3},
we can now deduce the following energy dissipation relation.

\begin{lemma}[Energy dissipation and port-Hamiltonian structure] \label{lem:energy} $ $ \\
Let $E(t) = \sum\nolimits_e E^e(t)$ denote the total acoustic energy contained in the network.
Then 
\begin{align*}
\frac{d}{dt} E(t) = -\sum\nolimits_{e\in\E} \int_e d^e |q^e(t)|^2 dx + \sum\nolimits_{v \in \Vb} u_v(t) y_v(t),
\end{align*}
i.e., the change of total energy is caused by power dissipated through the damping mechanism and supplied or drained at the system ports.
\end{lemma}
\begin{proof}
By elementary calculations and the partial differential equations \eqref{eq:sys1}--\eqref{eq:sys2}, we get
\begin{align*}
\frac{d}{dt} \frac{1}{2} \int_e a^e |p^e|^2 + b^e |q^e|^2 dx 
&= \int_e a^e \dt p^e p^e + b^e \dt q^e q^e dx \\
&= \int_e (-\dx q^e) p^e + (-\dx p^e - d^e q^e) q^e dx.
\end{align*}
Integration-by-parts of the second term in the last equation on $e=(v_1^e,v_2^e)$ gives 
\begin{align*}
\int_e (-\dx p^e) q^e dx = \int_e p^e \dx q^e dx - n^e(v_1^e) q^e(v_1^e) p^e(v_1^e) -  n^e(v_2^e) q^e(v_2^e) p^e(v_2^e).
\end{align*}
Summing over all edges $e$, using the definition of the total energy, and \eqref{eq:sys4}--\eqref{eq:sys3} leads to 
\begin{align*}
\frac{d}{dt} E(t) 
&= -\sum\nolimits_{e \in \E} \int_e d^e |q^e|^2 dx -  \sum\nolimits_{v \in \Vb} n^e(v) q^e(v) p^e(v).
\end{align*}
The result now follows from definition of the in- and output.
\end{proof}

\subsection{Exponential stability}
Due to linearity of the problem, it suffices to consider the homogeneous case.
The energy balance then reveals that kinetic energy is dissipated by the damping mechanism. 
This, however, also leads to a reduction of the total energy resulting in the exponential stability of the system stated as property (P4).  
\begin{lemma}[Exponential stability] \label{lem:stability} $ $ \\
Let $u(t) \equiv 0$ for $0 \le t_1 \le t \le t_2$, 
and $E(t)$ be the total energy of the system. 
Then 
\begin{align*}
E(t) \le C e^{-\gamma (t-s)} E(s) \qquad \text{for all } t_1 \le s \le t \le t_2,
\end{align*}
with positive constants $C,\gamma>0$ that are independent of $t_1$, $t_2$, $s$, and $t$.
\end{lemma}
\begin{proof}
The proof is based on energy estimates, some graph theoretic results, and a 
generalized Poincar\'e inequality; we refer to \cite{EggerKugler16b} for details.
\end{proof}

\subsection{Steady states}
From the previous result, we obtain convergence to zero steady state in case of homogeneous input $u \equiv 0$. 
Due to the linearity of the problem, this yields also the existence of unique and stable steady states in the general case.
\begin{lemma}[Steady states] \label{lem:equilibrium}
Let $u(t) \equiv const$ for all $t \ge t_1$. Then $(p(t),q(t))$ converges to 
a steady state $(\bar p, \bar q)$, which is the unique solution of the corresponding stationary problem.
\end{lemma}
\begin{proof}
The existence of a unique steady state has been established in \cite{EggerKugler16b}. 
The difference to steady state $(p(t)-\bar p,q(t)-\bar q)$ solves \eqref{eq:sys1}--\eqref{eq:sys5} with $u \equiv 0$
and convergence to steady state thus follows by the exponential stability estimate given in the previous lemma. 
\end{proof}

\section{Galerkin approximation} \label{sec:galerkin}

We now extend the discretization strategy proposed in \cite{EggerKugler16b} to our setting and review the basic results about the stability of these full order models.
Let $P_h \subset L^2(\E)$ and $Q_h \subset H^1(\E)$
be finite dimensional spaces and set $\Lambda_h = \RR^{\Vi}$. 
Let $T>0$ and consider the following conforming Galerkin approximations of the variational principle \eqref{eq:var1}--\eqref{eq:var3} as space discretization.
\begin{problem}[Galerkin approximation and discrete variational problem] \label{prob:galerkin} $ $\\
Find $p_h \in H^1(0,T;P_h)$, $q_h \in H^1(0,T;Q_h)$, and $\lambda_h \in L^2(0,T;\Lambda_h)$ 
such that 
\begin{align*}
(p_h(0),\tilde p_h)_\E = (p_0,\tilde p_h)_\E
\qquad \text{and} \qquad  
(q_h(0),\tilde q_h)_\E = (q_0, \tilde q_h)_\E
\end{align*}
for all $\tilde p_h \in P_h$ and $\tilde q_h \in Q_h$, and such that the discrete variational equations
\begin{align}
(a \dt p_h(t), \tilde p_h)_\E + (\dx' q_h(t), \tilde p_h)_\E &= 0, \label{eq:var1_disc}\\
(b \dt q_h(t), \tilde q_h)_\E - (p_h(t), \dx' \tilde q_h)_\E + (d q_h(t),\tilde q_h)_\E + (\lambda_h(t), [n \tilde q_h])_{\Vi} &= (u(t),n \tilde q_h)_{\Vb}, \label{eq:var2_disc}\\
([n q_h(t)],\tilde \lambda_h)_{\Vi} &= 0 \label{eq:var3_disc}
\end{align}
hold for all test functions $\tilde p_h \in P_h$, $\tilde q_h \in Q_h$, $\tilde \lambda_h \in \Lambda_h$, and all $0 \le t \le T$. 
\end{problem}
\noindent
A simple compatibility condition allows to deduce the well-posedness of this problem.
\begin{lemma}[Discrete well-posedness] $ $\\
Assume that $\{1^e : e \in \E_h\} \subset Q_h$, where $1^e$ denotes the function in $L^2(\E)$ which is constant one on the edge $e$ and zero otherwise.
Then Problem~\ref{prob:galerkin} has a unique solution. 
\end{lemma}
\begin{proof}
The condition $\{1^e : e \in \E\} \subset Q_h$ allows to eliminate the constraint and the Lagrange multiplier $\lambda$ from the system; compare with Section~\ref{sec:ode}. 
By choosing bases for the spaces $P_h$ and $Q_h$, we may thus obtain linear system of ordinary differential equations. Existence and uniqueness then follow from the Picard-Lindelöf theorem.
\end{proof}

\begin{remark} \label{rem:reduction}
Once the solution of the discretized problem is found, 
we can define the corresponding output of the discrete system  by $y_h(v) = -n^e(v) q^e_h(v)$ for $v \in \Vb$ and $e \in \E(v)$. 
Let us note that elimination of the constraints \eqref{eq:var3_disc} yields the problem originally considered in \cite{EggerKugler16b}. All results obtained in that paper therefore carry over to the problem considered here, if the condition of the previous Lemma is satisfied, which is called assumption (A3$_h$) below. 
\end{remark}

We next establish the properties (P1)--(P4) for the Galerkin approximations introduced above. The results again follow with similar arguments as used in \cite{EggerKugler16b}. Let us emphasize that additional conditions on the approximation spaces $P_h$ and $Q_h$ are required for some of the results.

\subsection{Conservation of mass}

Mass conservation on the continuous level follows by testing \eqref{eq:var1} with the function $\tilde p \equiv 1$
and some elementary manipulations. A discrete equivalent of this result can be obtained, 
if $\tilde p_h \equiv 1$ is contained in the test space.

\begin{lemma}[Discrete mass conservation] \label{lem:mass_disc}$ $\\
Let $m_h(t) = \sum\nolimits_{e \in \E} \int_e a p_h(t) dx$ denote the total mass of the discrete system, 
and let
\begin{itemize}\itemsep1ex
 \item[(A1$_h$)] $1 \in P_h$.
\end{itemize}
Then the total mass changes only due to flux across the boundary, i.e., 
\begin{align*}
\frac{d}{dt} m_h(t) = \sum\nolimits_{v \in \Vb} -n(v) q_h(v) = \sum\nolimits_{v \in \Vb} y_h(v).
\end{align*}
\end{lemma}
\begin{proof}
The assertion follows in the same manner as that of Lemma~\ref{lem:mass}.
\end{proof}

\subsection{Energy balance}
Mimicking the notation used on the continuous level, 
we may define for every pipe $e \in \E$ the total (discrete) energy content of the pipe by
\begin{align*}
E^e_h(t) = \frac{1}{2} \int_e a^e |p_h(t)|^2 + b^e |q_h(t)|^2.
\end{align*}
With the same arguments as on the continuous level, we then obtain
\begin{lemma}[Discrete energy balance and port-Hamiltonian structure] \label{lem:energy_disc}$ $\\
Let $E_h(t) = \sum\nolimits_{e \in \E} E_h^e(t)$ denote the total discrete energy. Then 
\begin{align} \label{eq:endish}
\frac{d}{dt} E_h(t) = -\sum\nolimits_{e \in \E} \int_e a^e |q_h(t)|^2 dx + \sum\nolimits_{v \in \Vb} y_h(t) u(t),
\end{align} 
i.e., the energy changes by dissipation and supply or drain via the ports of the network.
\end{lemma}
Note that no extra condition for the approximation spaces is required for the proof of property (P1)
for the Galerkin approximations of the variational formulation \eqref{eq:var1}--\eqref{eq:var3}.

\subsection{Exponential stability}
For input $u \equiv 0$, the energy balance \eqref{eq:endish} already guarantees that the total energy of the discrete system is non-increasing. 
Under additional compatibility conditions on the approximation spaces, one can even show 
the uniform exponential decay.
\begin{lemma}[Uniform discrete exponential stability] \label{lem:stability_disc} 
Assume that 
\begin{itemize}\itemsep1ex
 \item[(A2$_h$)] $\dx' Q_h = P_h$; 
 \item[(A3$_h$)] $1^e \in Q_h$ for all $e \in \E$.
\end{itemize}
Then for $u \equiv 0$, the discrete energy $E_h(t)$ defined in Lemma~\ref{lem:energy_disc} satisfies 
\begin{align*}
E_h(t) \le C e^{-\gamma(t-s)} E_h(s) , \qquad t \ge s \ge 0.
\end{align*}
Moreover, the constants $C,\gamma$ can be chosen the same as on the continuous level.
\end{lemma}
\begin{proof}
Following Remark~\ref{rem:reduction}, the proof can be deduced from the results of \cite{EggerKugler16b}. 
\end{proof}

\subsection{Steady states}
Under the assumptions of the previous lemma, one can also guarantee property (P4), i.e., the 
existence and uniqueness of discrete steady states.
\begin{lemma}[Discrete steady states] \label{lem:steady_disc} $ $\\
Let $u(t) \equiv const$ for $t \ge t_0$ and assume that (A2$_h$)--(A3$_h$) hold.
Then for $t \to \infty$, the discrete solution $(p_h(t),q_h(t))$ converges to a discrete equilibrium
 $(\bar p_h,\bar q_h)$ which is the unique solution of the corresponding stationary problem.
\end{lemma}
\begin{proof}
The existence of a unique discrete steady state is established in \cite{EggerKugler16b}.
Convergence to equilibrium then follows from the energy decay estimate like in Lemma~\ref{lem:equilibrium}.
\end{proof}

\subsection{A mixed finite element approximation} \label{sec:fem}

As a particular Galerkin approximation satisfying the above assumptions, let us briefly discuss the mixed finite element method that is used in our numerical tests.
Let $[0,l^e]$ be the interval represented by the edge $e$
and denote by $T_h(e) = \{T\}$ a uniform mesh of $e$ with subintervals $T$ of length $h^e$. The global mesh is then defined as $T_h(\E) = \{T_h(e) : e \in \E\}$, 
and the global mesh size is denoted by $h=\max_e h^e$. 
We denote the spaces of piecewise polynomials on $T_h(\E)$ by
\begin{align*}
 P_k(T_h(\E)) &= \{v \in L^2(\E) : v|_e \in P_k(T_h(e)), \ e \in \E\},
\end{align*}
where $P_k(T_h(e)) = \{v \in L^2(e) : v|_T \in P_k(T), \ T \in T_h(e)\}$ and $P_k(T)$ 
is the space of polynomials of degree $\le k$ on the subinterval $T$. 
Note that $P_k(T_h(\E)) \subset L^2(\E)$, which is easy to see, 
but in general $P_k(T_h(\E)) \not\subset H^1(\E)$. 
As spaces $V_h$ and $Q_h$ for the Galerkin approximation presented in the previous sections, we now consider
\begin{align} \label{eq:spaces}
V_h = P_{1}(T_h(\E)) \cap H(\div)
\quad \text{and} \quad 
Q_h = P_{0}(T_h(\E)). 
\end{align}
This choice of spaces satisfies the compatibility conditions (A1$_h$)--(A3$_h$); see \cite{EggerKugler16b} for details.

\subsection{Structure preserving model reduction} \label{sec:strucpresmodelred}

As final result of this section, we now give an interpretation of the model reduction approach on the level of function spaces. The system obtained by Galerkin projection onto $P_h$, $Q_h$ will again be called full order model. The reduced models are obtained by projection onto smaller subspaces $P_H \subset P_h$ and $Q_H \subset Q_h$.
The compatibility conditions for these coarse subspaces read
\begin{itemize}\itemsep1ex
 \item[(A1$_H$)] $1 \in P_H$;
 \item[(A2$_H$)] $\dx' Q_H = P_H$;
 \item[(A3$_H$)] $1^e \in Q_H$ for all $e \in \E$.
\end{itemize}
Note that by construction $P_H \subset P_h \subset L^2(\E)$ and $Q_H \subset Q_h \subset H^1(\E)$.
From the previous results about general Galerkin approximations, we therefore directly deduce the following result.
\begin{lemma}[Structure preserving model reduction] $ $\\
Let $P_H \subset P_h$ and $Q_H \subset Q_h$ and assume that (A1$_H$)--(A3$_H$) hold. 
Then the reduced system satisfies (P1)--(P4) and the assertions of Lemma~\ref{lem:energy_disc}--\ref{lem:steady_disc} hold accordingly.
\end{lemma}

\begin{remark}
Since the reduced model can be viewed as Galerkin approximation of the infinite dimensional  problem~\eqref{eq:var1}--\eqref{eq:var3}, it is clear that the solution $(p_H(t),q_H(t))$ only depends on the choice of the approximation spaces $P_H$ and $Q_H$ but not on the spaces $P_h$ and $Q_h$ used for the generation of the full order model which is only required for computational purposes and has no effect on the quality of the reduced model.
\end{remark}

\section{Reformulation on the algebraic level} \label{sec:algebraic}

We now translate the results of the previous section to the algebraic level. 
By choosing appropriate bases for the subspaces $P_h$ and $Q_h$ defining 
the discrete variational problem \eqref{eq:var1_disc}--\eqref{eq:var3_disc},
the resulting full order model can be written in algebraic form as follows.
\begin{lemma}[Equivalent algebraic system] \label{lem:algebraic} $ $\\
Let $\{\phi_i\}$ and $\{\psi_j\}$ be bases for $P_h$ and $Q_h$. 
Then the problem \eqref{eq:var1_disc}--\eqref{eq:var3_disc} is equivalent to 
the system \eqref{eq:lti1}--\eqref{eq:lti3} with matrices defined by $M_1(i,j)=(a \phi_j, \phi_i)_\E$, $M_2(i,j)=(b \psi_j,\psi_i)_\E$, $G(i,j)=(\dx' \psi_j,\phi_i)_\E$, $D(i,j)=(d \psi_j,\psi_i)_\E$, $N(i,j)=[n \psi_j](v_{i_0})$, and $B_2(i,j)=-\delta_{j_\partial,i} n(v_{j}) \psi_i(v_{j})$. 
\end{lemma}
Here  $(a,b)_\E = \sum\nolimits_e \int_e a(x) b(x) dx$ is the scalar product of $L^2(\E)$, 
and $1 \le i_0 \le |\Vi|$ and $1 \le i_{\partial} \le |\Vb|$ denote the appropriate renumbering of inner and boundary vertices.
From the definition of the matrices, we also obtain the following properties. 
\begin{lemma}
Let $a^e$, $b^e$, and $d^e$ be positive for all $e \in \E$. Then $M_1$, $M_2$, $D$ are symmetric and positive definite.  
If (A2$_h$)--(A3$_h$) hold, then $[G^\top,N^\top]$ is injective and $N$ is surjective on $\N(G)$.
\end{lemma}

These properties allow us to establish the well-posedness of the linear time-invariant system \eqref{eq:lti1}--\eqref{eq:lti3} and, following our discussion in Section~\ref{sec:basic_red}, 
also the unique solvability of the corresponding stationary problem. 
As a next step, we can now provide an interpretation of the properties (P1)--(P4) on the algebraic level.

\subsection{Conservation of mass}

The condition $1 \in P_h$ is equivalent to 
\begin{itemize}\itemsep1ex
 \item[(A1$_h'$)] $\exists o_1 \in \RR^{k_1} : \sum\nolimits_{i=1}^{k_1} o_{1,i} \phi_i = 1$ on $\E$. 
\end{itemize}
Here $o_1 \in \RR^{k_1}$ is the coordinate vector representing the function $1 \in P_h$ 
in the basis $\{\phi_i\}$ of the space $P_h$. 
This allows us to express the conservation of mass on the algebraic level as follows.
\begin{lemma}[Mass conservation] 
Let $p_h(\cdot,t) = \sum\nolimits_{i=1}^{k_1} x_{1,i}(t) \phi_i(\cdot) \subset P_h$
and further define $q_h(\cdot,t) = \sum\nolimits_{j=1}^{k_2} x_{2,j}(t) \psi_j(\cdot) \subset Q_h$.
Then the mass of the discrete system can be expressed as 
\begin{align*}
m_h(t) 
= \sum\nolimits_{e \in \E} \int_e a^e p_h^e dx
= o_1^\top M_1 x_1(t).
\end{align*}
With $\hat o \in \RR^{\Vb}$ denoting the constant one vector and $y=B_2^\top x_2$ the output, we have
\begin{align*}
\frac{d}{dt} m_h(t) = -\hat o^\top y. 
\end{align*}
\end{lemma}
\begin{proof}
The definition of the  total mass and \eqref{eq:lti1} lead to
\begin{align*}
\frac{d}{dt} m_h 
= o^\top_1 M_1 \dot x_1 
= -o^\top_1 G x_2. 
\end{align*}
The fact that $o^\top_1 G x_2 = \hat o^\top B x_2 = \hat o^\top y$ 
can be deduced from the equivalent formulation of the Galerkin approximation in function spaces; 
cf. Lemma~\ref{lem:mass_disc}. 
\end{proof}
Again, the definition of the total mass and the proof of the mass conservation 
requires awareness of the underlying problem in function spaces.

\subsection{Energy balance}

The energy of the discrete system can be expressed as 
\begin{align*}
E_h(t) 
= \frac{1}{2}\sum\nolimits_{e \in \E} \big( a^e \|p_h^e(t)\|_{\E}^2 + b^e \|q_h^e(t)\|^2_\E \big) 
= \frac{1}{2} \big( x_1(t)^\top M_1 x_1(t) + x_2(t)^\top M_2 x_2(t) \big),
\end{align*}
where $(x_1(t),x_2(t),x_3(t))$ is a solution of \eqref{eq:lti1}--\eqref{eq:lti3} and $p_h(\cdot,t) = \sum\nolimits_{i=1}^{k_1} x_{1,i}(t) \phi_i(\cdot) \subset P_h$, $q_h(\cdot,t) = \sum\nolimits_{j=1}^{k_2} x_{2,j}(t) \phi_j(\cdot)$, and $\lambda_h(t)=x_3(t)$ define the corresponding functions making up the solution of of Problem~\ref{prob:galerkin}.  
The discrete energy balance of Lemma~\ref{lem:energy_disc} can now be rephrased as
\begin{lemma}[Energy dissipation and port-Hamiltonian structure] \label{lem:energy_alg}$ $\\
Let $(x_1,x_2,x_3)$ denote a solution of the linear system \eqref{eq:lti1}--\eqref{eq:lti3}. 
Then 
\begin{align*}
\frac{d}{dt} E_h(t) = -x_2(t)^\top D x_2(t) + y(t)^\top u(t),
\end{align*}
with output defined as $y(t)=B^\top x_2(t)$. In particular, property (P1) is valid.
\end{lemma}
\begin{proof}
Let us give a direct derivation of this assertion on the algebraic level. 
Using the definition of the energy, the symmetry of $M_i$, and the algebraic equations, we obtain 
\begin{align*}
\frac{d}{dt} E_h 
&= x_1^\top M_1 \dot x_1 + x_2^\top M_2 \dot x_2 \\
&= x_1^\top (-G x_2) + x_2^\top (G^\top x_1 - D x_2 + B u) 
 = -x_2^\top D x_2 + y^\top u. 
\end{align*}
In the last step, we utilized that $x_2^\top B u = (B^\top x_2)^\top u$ and the definition of the output.
\end{proof}

\begin{remark}
As can be seen from the proof, the scalar product and norm induced by the matrices $M_1$ and $M_2$ are directly associated with the energy of the system and they are the natural ones for the analysis and numerical treatment of the discrete problem in algebraic form.
\end{remark}

As shown above, the property (P1) follows directly from the particular form of the algebraic system. As will become clear below, the validity of the remaining properties (P2)--(P4) however requires awareness of the underlying infinite dimensional problem.

\subsection{Exponential stability}

In order to translate the assertions of Lemma~\ref{lem:stability_disc} to the algebraic level,
we have to describe the meaning of the conditions (A2$_h$)--(A3$_h$).
\begin{lemma}
The compatibility conditions  (A2)--(A3) are equivalent to 
\begin{itemize}\itemsep1ex
 \item[(A2$_h'$)] for all $x_1 \in \RR^{k_1}$ there exists $x_2 \in \RR^{k_2}$ such that  $\sum\nolimits_{i=1}^{k_1} x_{1,i} \phi_i = \sum\nolimits_{j=1}^{k_2} x_{2,j} \psi_j$;
 \item[(A3$_h'$)] for all $e \in \E$ there exists $o_2^e \in \RR^{k_2}$ such that $\sum\nolimits_{j=1}^{k_2} o^e_{2,j} \psi_j = 1^e \in Q_h$.
\end{itemize}
\end{lemma}
As a direct consequence of this characterization, Lemma~\ref{lem:stability_disc}, 
and the equivalence of the algebraic system to the Galerkin approximation, 
we obtain the following lemma.
\begin{lemma}[Exponential stability] $ $\\
Let (A2$_h'$)--(A3$_h'$) hold 
and let $(x_1(t),x_2(t),x_3(t))$ be a solution of \eqref{eq:lti1}--\eqref{eq:lti3}. 
Then 
\begin{align*}
\frac{d}{dt} E_h(t) \le C e^{-\gamma(t-s)} E_h(s), 
\end{align*}
with constants $C,\gamma>0$ that can be chosen as in Lemma~\ref{lem:stability} and \ref{lem:stability_disc}.
\end{lemma}

The derivation of the conditions (A2$_h'$)--(A3$_h'$) and the proof of the stability estimate can thus be deduced from the underlying Galerkin approximation and the analysis in function spaces.

\subsection{Steady states}

The assertion about steady states can finally be translated as follows.
\begin{lemma}[Steady states] 
Let (A2$_h'$)--(A3$_h'$) hold. 
Then for $u(t) \equiv const$, the solutions $(x_1(t),x_2(t),x_3(t))$ of the system \eqref{eq:lti1}--\eqref{eq:lti3}
converge to a steady state $(\bar x_1,\bar x_2,\bar x_3)$ which is the unique solution of the corresponding stationary problem. 
\end{lemma}
\begin{proof}
The result follows again by equivalence to the Galerkin approximation \eqref{eq:var1_disc}--\eqref{eq:var3_disc} 
and the corresponding result stated in Lemma~\ref{lem:steady_disc}. 
\end{proof}

\subsection{Structure preserving model reduction} \label{sec:algebraicmodred}

As a final step of our analysis, we can now provide a proof for Theorem~\ref{thm:main} by translating the results of Section~\ref{sec:strucpresmodelred} to the algebraic level: 
Let $\{\phi_i\}$ and $\{\psi_j\}$ denote bases for $P_h$ and $Q_h$,
and let $V_1 \in \RR^{k_1 \times K_1}$ and $V_2 \in \RR^{k_2 \times K_2}$ be given matrices with linearly independent columns. 
For $k=1,\ldots,K_1$ and $l=1,\ldots,K_2$, we define
\begin{align*}
\Phi_k = \sum\nolimits_{i=1}^{k_1} V_{1,ki} \phi_i, 
\qquad \text{and} \qquad 
\Psi_l = \sum\nolimits_{j=1}^{k_2} V_{2,lj} \psi_j,
\end{align*}
which serve as basis functions for low dimensional approximation spaces $P_H$ and $Q_H$.
As a direct consequence of this construction and the previous considerations, we obtain 
\begin{lemma}[Reduced algebraic system]
Let us define
\begin{align*}
P_H = \text{span}\{\Phi_k : k=1,\ldots,K_1\}
\qquad \text{and} \qquad 
Q_H = \text{span}\{\Psi_l : l=1,\ldots,K_2\}.
\end{align*} 
Then $P_H \subset P_h$ and $Q_H \subset Q_h$, and the corresponding discrete variational problem is equivalent to the reduced order system
\eqref{eq:red1}--\eqref{eq:red3} with matrices as defined in Section~\ref{sec:spmr}.
\end{lemma}
The property (P1) for the reduced system \eqref{eq:red1}--\eqref{eq:red3} again follows directly from the special algebraic form of the reduced problem. In order to guarantee (P2)--(P4), we require additional compatibility conditions. The following characterization clarifies the picture. 
\begin{lemma}[Algebraic compatibility conditions] $ $\\
Let (A1$_h'$)--(A3$_h'$) hold and assume that the algebraic compatibility conditions 
(A1)--(A3) are valid. 
Then $P_H \subset P_h$ and $Q_H \subset Q_h$ satisfy the compatibility conditions (A1$_H$)--(A3$_H$). 
\end{lemma}
As a direct consequence of the previous considerations, we now obtain the following 
result.
\begin{lemma}[Structure preserving model reduction] $ $\\
Let (A1$_h'$)--(A3$_h'$) hold for the system \eqref{eq:lti1}--\eqref{eq:lti3} 
and assume that the algebraic conditions (A1)--(A3) are valid.
Then the reduced system \eqref{eq:red1}--\eqref{eq:red3} satisfies (P1)--(P4). 
\end{lemma}
\noindent
This lemma yields a correct statement of Theorem~\ref{thm:main} and 
completes the proof of our assertions.

\end{document}